\documentclass[a4paper,12pt]{amsart}
\usepackage{amsfonts}
\usepackage{amsthm}
\usepackage{amssymb}
\usepackage{amsmath}
\usepackage{enumerate}
\usepackage{url}
\usepackage{color}
\usepackage{comment}
\usepackage{a4wide}
\usepackage{dsfont}
\usepackage{bbm}
\usepackage{todonotes}
\usepackage{wasysym}

\newtheorem{thm}{Theorem}
\newtheorem{prop}[thm]{Proposition}

\newtheorem{lemma}[thm]{Lemma}

\theoremstyle{definition}

\theoremstyle{remark}

\allowdisplaybreaks

\begin{document}
	\title[On the metric theory of approximations by reduced fractions]{On the metric theory of approximations by reduced fractions: a quantitative Koukoulopoulos--Maynard theorem}
	\author{Christoph Aistleitner}
	\author{Bence Borda}
	\author{Manuel Hauke}
	\address{Graz University of Technology, Institute of Analysis and Number Theory, Steyrergasse 30/II, 8010 Graz, Austria}
	\email{aistleitner@math.tugraz.at}
	\email{borda@math.tugraz.at}
	\email{hauke@math.tugraz.at}
	
	
	\renewcommand{\labelenumi}{\alph{enumi})}
	
	\newcommand{\mods}[1]{\,(\mathrm{mod}\,{#1})}
	
	\begin{abstract} 
		Let $\psi: \mathbb{N} \to [0,1/2]$ be given. The Duffin--Schaeffer conjecture, recently resolved by Koukoulopoulos and Maynard, asserts that for almost all reals $\alpha$ there are infinitely many coprime solutions $(p,q)$ to the inequality $|\alpha - p/q| < \psi(q)/q$, provided that the series $\sum_{q=1}^\infty \varphi(q) \psi(q) / q$ is divergent. In the present paper, we establish a quantitative version of this result, by showing that for almost all $\alpha$ the number of coprime solutions $(p,q)$, subject to $q \leq Q$, is of asymptotic order $\sum_{q=1}^Q 2 \varphi(q) \psi(q) / q$. The proof relies on the method of GCD graphs as invented by Koukoulopoulos and Maynard, together with a refined overlap estimate coming from sieve theory, and number-theoretic input on the ``anatomy of integers''. The key phenomenon is that the system of approximation sets exhibits ``asymptotic independence on average'' as the total mass of the set system increases.
	\end{abstract}
	
	\maketitle
	
	\section{Introduction and statement of results}
	
	A foundational result in Diophantine approximation is Dirichlet's approximation theorem, which asserts that for every real number $\alpha$ there are infinitely many coprime solutions $(p,q)$ to the inequality
	\begin{equation} \label{diri}
		\left| \alpha - \frac{p}{q} \right| < \frac{1}{q^2}.
	\end{equation}
	It is well-known that this result is optimal up to constant factors for numbers $\alpha$ whose partial quotients in the continued fraction representation are bounded (so-called badly approximable numbers). Metric number theory asks to what extent \eqref{diri} can be improved for typical reals $\alpha$, in the sense that the exceptional set has vanishing Lebesgue measure.

	One of the fundamental results of metric Diophantine approximation is Khintchine's theorem~\cite{khin}. Let $\psi(q)$ be a non-negative sequence, and suppose that $q \psi(q)$ is non-increasing. Then the inequality
	\begin{equation} \label{khin}
		\left| \alpha - \frac{p}{q} \right| < \frac{\psi(q)}{q}
	\end{equation}
	has infinitely many integer solutions $(p,q)$ for almost all real numbers $\alpha$, provided that the series $\sum_{q=1}^\infty \psi(q)$ diverges. In contrast, inequality \eqref{khin} has only finitely many solutions for almost all $\alpha$ if this series converges. Very roughly speaking, this says that for typical reals the Dirichlet approximation theorem can be improved by a factor of logarithmic order. By periodicity, it is sufficient to consider $\alpha \in [0,1]$. It can easily be seen that Khintchine's theorem addresses the question whether the set system
	$$
	\bigcup_{p=0}^q \left( \frac{p}{q} - \frac{\psi(q)}{q}, \frac{p}{q} + \frac{\psi(q)}{q} \right) \cap [0,1], \qquad q=1,2,\dots,
	$$
	contains a given real $\alpha$ for infinitely resp.\ only finitely many values of $q$. If we assume that $\psi(q) \leq 1/2$ (as we will throughout this paper, to avoid degenerate situations), then the measure of such a set is exactly $2\psi(q)$. Thus the ``only finitely many'' part of Khintchine's theorem is a straightforward application of the convergence part of the Borel--Cantelli lemma. The ``infinitely many'' part of the theorem, however, is much more delicate since the divergence part of the Borel--Cantelli lemma requires some form of stochastic independence. The purpose of the monotonicity condition in the statement of Khintchine's theorem is to guarantee this stochastic independence property of the set system.
	
	Duffin and Schaeffer~\cite{duff} showed that Khintchine's theorem generally fails without the monotonicity condition. More precisely, they constructed a function $\psi$ which is supported on a set of very smooth integers (having a large number of small prime factors), such that $\sum_{q=1}^{\infty} \psi(q)$ diverges, but for almost all $\alpha$ there are only finitely many solutions to \eqref{khin}. From a probabilistic perspective, the counterexample of Duffin and Schaeffer exploits the lack of stochastic independence in the set system, by constructing a special configuration where the overlaps between different sets of the system are too large; the crucial point here is that a fraction $p/q$ can have many different representations as a quotient of integers (as long as non-reduced representations are allowed), and thus may appear in many different elements of the set system.

Duffin and Schaeffer suggested that this lack of independence could be overcome by switching to the coprime setting. More precisely, the Duffin--Schaeffer conjecture asserted that for almost all $\alpha$ there are infinitely many \emph{coprime} solutions $(p,q)$ to \eqref{khin} if and only if the series $\sum_{q=1}^\infty \varphi(q) \psi(q)/q$ diverges, where $\varphi$ denotes the Euler totient function. Let 
	\begin{equation} \label{eq_def}
		\mathcal{A}_q := \bigcup_{\substack{1 \leq p \leq q,\\ \gcd(p,q)=1}} \left( \frac{p}{q} - \frac{\psi(q)}{q}, \frac{p}{q} + \frac{\psi(q)}{q} \right), \qquad q=1,2,\ldots .
	\end{equation}
	Then, writing $\lambda$ for the Lebesgue measure and again assuming that $\psi(q) \leq 1/2$ for all $q$, we have 
	$$
	\lambda (\mathcal{A}_q ) = \frac{2 \varphi(q) \psi(q)}{q}.
	$$
	Thus the ``only finitely many'' part of the Duffin--Schaeffer conjecture is again a direct consequence of the convergence part of the Borel--Cantelli lemma. However, the divergence part of the Duffin--Schaeffer conjecture has resisted a resolution for many decades. After important contributions of Gallagher~\cite{gall}, Erd\H os~\cite{erd}, Vaaler~\cite{vaal}, Pollington and Vaughan~\cite{pv}, and Beresnevich and Velani~\cite{ber_ve}, the Duffin--Schaeffer conjecture was finally solved in full generality by Koukoulopoulos and Maynard~\cite{km} in 2020. Their argument relies on an ingenious construction of what they call ``GCD graphs''. This allows them to implement a step-by-step quality increment strategy until they finally arrive at a situation where they can completely control the divisor structure which is at the heart of the problem. The final, number-theoretic input, is an ``anatomy of integers'' statement that quantifies the observation that there are only few integers that have many small prime factors. 
	
	In the present paper, we prove a quantitative version of the Koukoulopoulos--Maynard theorem. Their result states that there are infinitely many coprime solutions to \eqref{khin} for almost all $\alpha$ if the sum of measures diverges. We show that for almost all $\alpha$ the number of solutions in fact grows proportionally to the sum of measures.	
	\begin{thm} \label{th1}
		Let $\psi:~\mathbb{N} \to [0,1/2]$ be a function such that $\sum_{q=1}^{\infty} \frac{\varphi(q) \psi(q)}{q}=\infty$. Write $S(Q)=S(Q,\alpha)$ for the number of coprime solutions $(p,q)$ to the inequality
		$$
		\left| \alpha - \frac{p}{q} \right| < \frac{\psi(q)}{q}, \qquad \text{subject to $q \leq Q$},
		$$
		and let
		\begin{equation} \label{psiq_def}
			\Psi(Q) = \sum_{q=1}^Q \frac{2 \varphi(q) \psi(q)}{q}.
		\end{equation}
		Let $C>0$ be arbitrary. Then for almost all $\alpha$,
		$$
		S(Q) = \Psi(Q) \left( 1 + O \left(\frac{1}{(\log \Psi(Q))^{C}}\right) \right) \qquad \text{as } Q \to \infty .
		$$
	\end{thm}
	
	It is not clear to what extent the error term in the theorem can be improved. It seems to us that any result which contains a power saving, i.e.\ has a multiplicative error of order $(1 + O(\Psi(Q)^{-\varepsilon}))$ for some $\varepsilon>0$, would require a substantial improvement of the argument in the present paper. By analogy with other results from metric number theory it is reasonable to assume that Theorem~\ref{th1} actually holds with an error term $(1 + O(\Psi(Q)^{-1/2+\varepsilon}))$ for any $\varepsilon>0$, and probably even $(1 + O(\Psi(Q)^{-1/2} (\log \Psi(Q))^c))$ for some appropriate $c$. We note in passing that very precise metric estimates for the asymptotic order of $S(Q)$ are known when an extra monotonicity assumption is imposed upon $\psi$, in the spirit of Khintchine's original result; see for example Chapter 3 of~\cite{phil} and Chapter 4 of~\cite{harman}. However, from a technical perspective, the problem is of a very different nature when this extra monotonicity assumption is made. The results for the monotonic case imply as a corollary that Theorem~\ref{th1} above cannot hold in general with a multiplicative error of order $(1 + O(\Psi(Q)^{-1/2}))$ or less.
	
	The key problem in the metric theory of approximations by reduced fractions is to control the measure of the overlaps $\mathcal{A}_q \cap \mathcal{A}_r$ in some averaged sense. Pairwise independence $\lambda (\mathcal{A}_q \cap \mathcal{A}_r) = \lambda(\mathcal{A}_q) \lambda(\mathcal{A}_r)$ would allow a direct application of the second Borel--Cantelli lemma, but it turns out that $\lambda(\mathcal{A}_q \cap \mathcal{A}_r)$ can exceed $\lambda(\mathcal{A}_q) \lambda(\mathcal{A}_r)$ by a factor as large as $\log \log (qr)$ for some configurations of $q,r,\psi$. Such an exceedingly large overlap can happen if there are many small prime factors dividing $q$ but not dividing $r$, or vice versa, and if simultaneously the greatest common divisor of $q$ and $r$ lies in a certain critical range (which is determined by the values of $\psi(q)$ and $\psi(r)$). The crucial point then is to show that such large extra factors appear only for a small number of pairs $q,r$. Consider the quotient
	$$
	\frac{\sum_{q,r \leq Q} \lambda( \mathcal{A}_q \cap \mathcal{A}_r)}{\left( \sum_{q=1}^Q \lambda(\mathcal{A}_q) \right)^2} = \frac{\sum_{q,r \leq Q} \lambda( \mathcal{A}_q \cap \mathcal{A}_r)}{\Psi(Q)^2}.
	$$
	Without imposing an absolute lower bound on $\Psi(Q)$, this quotient can be arbitrarily large. The main breakthrough of Koukoulopoulos and Maynard was to prove that 
	\begin{equation}\label{KMtechnical}
	\frac{\sum_{q,r \leq Q} \lambda( \mathcal{A}_q \cap \mathcal{A}_r)}{\Psi(Q)^2} \ll 1 \qquad \text{provided that} \quad \Psi(Q) \geq 1.
	\end{equation}
	This property is called \emph{quasi-independence on average}, and is sufficient for an application of the second Borel--Cantelli lemma (in the Erd\H os--R\'enyi formulation of the lemma) --  cf.\ \cite{bv}. In the present paper, we show that even more is true: we have
	$$
	\frac{\sum_{q,r \leq Q} \lambda( \mathcal{A}_q \cap \mathcal{A}_r)}{\Psi(Q)^2} \to 1 \qquad \text{as} \quad \Psi(Q) \to \infty.
	$$
	Thus the set system $(\mathcal{A}_q)_{q \geq 1}$ moves towards pairwise independence on average as the total mass of the set system (the sum of measures of the approximation sets) tends towards infinity. Since we consider this fact, which is the key ingredient in our proof of Theorem~\ref{th1}, to be very interesting in its own right, we state it below as a separate theorem.
	
	\begin{thm} \label{th2}
		Let $\psi:~\mathbb{N} \to [0,1/2]$ be a function. Let the sets $\mathcal{A}_q$, $q=1,2,\dots$, be defined as in \eqref{eq_def}, and let $\Psi(Q)$ be defined as in \eqref{psiq_def}. Let $C>0$ be arbitrary. For any $Q \in \mathbb{N}$ such that $\Psi(Q) \ge 2$, we have
		$$
		\sum_{q,r \leq Q} \lambda( \mathcal{A}_q \cap \mathcal{A}_r) - \Psi(Q)^2 = O \left(\frac{\Psi(Q)^2}{(\log \Psi(Q))^{C}} \right)
		$$
		with an implied constant depending only on $C$.
	\end{thm}

The rest of this paper is organized as follows. In Section~\ref{sec_2}, we show how Theorem~\ref{th2} implies Theorem~\ref{th1}. The following seven sections are concerned with the proof of Theorem~\ref{th2}. Section~\ref{sec_3} contains an estimate for the measure of the overlap $\mathcal{A}_q \cap \mathcal{A}_r$ for given $q$ and $r$. This estimate exploits information on the divisor structure of $q$ and $r$ in order to bound the difference between $\lambda \left(\mathcal{A}_q \cap \mathcal{A}_r\right)$ and $\lambda (\mathcal{A}_q) \lambda(\mathcal{A}_r)$, thus addressing the issue of the ``stochastic dependence'' between $\mathcal{A}_q$ and $\mathcal{A}_r$. In Section~\ref{sec_4}, we reduce Theorem~\ref{th2} to two second moment bounds. Section~\ref{sec_5} contains a brief introduction to the ``GCD graph'' machinery developed by Koukoulopoulos and Maynard~\cite{km}. In Section~\ref{sec_6}, we show how the second moment bounds follow from the existence of a ``good'' GCD subgraph. In the final two sections, we establish the existence of such a good GCD subgraph, using a modification of the iteration procedure of~\cite{km}. Our argument requires a careful balancing of the ``quality gain'' against the potential ``density loss'' coming from this iterative procedure, in such a way that information on the ``anatomy of integers'' can be exploited beyond a certain threshold. This threshold is determined by the order of the error terms coming from sieve theory (which translate into the error terms of the overlap estimate in Section~\ref{sec_3}).

For the rest of the paper, $\psi: \mathbb{N} \to [0,1/2]$ is an arbitrary function, $\mathcal{A}_q$, $q \in \mathbb{N}$, is as in \eqref{eq_def}, and $\Psi(Q)$, $Q \in \mathbb{N}$, is as in \eqref{psiq_def}.

\section{Proof of Theorem 1} \label{sec_2}

Let $C>4$ be fixed, and assume that Theorem~\ref{th2} holds. Let $\mathbbm{1}_A$ denote the indicator function of a set $A$. Formulated in probabilistic language, Theorem~\ref{th2} controls the variance of the random variables $\mathbbm{1}_{\mathcal{A}_1}, \dots, \mathbbm{1}_{\mathcal{A}_Q}$, and we obtain
\begin{eqnarray} \label{vari}
\int_0^1 \left(\sum_{q=1}^Q \mathbbm{1}_{\mathcal{A}_q} (\alpha) - \Psi(Q) \right)^2 \mathrm{d}\alpha
& = & \sum_{q,r \leq Q} \lambda( \mathcal{A}_q \cap \mathcal{A}_r) - \Psi(Q)^2 = O \left(\frac{\Psi(Q)^2}{(\log \Psi(Q))^{C}} \right).
\end{eqnarray}
Define 
$$
Q_k = \min  \left \{Q:~\Psi(Q) \geq e^{k^{1/\sqrt{C}}} \right \}, \qquad k \geq 1,
$$
and let 
$$
\mathcal{B}_k = \left\{ \alpha \in [0,1]:~\left| \sum_{q=1}^{Q_k} \mathbbm{1}_{\mathcal{A}_q} (\alpha) - \Psi(Q_k) \right| \geq \frac{\Psi(Q_k)}{(\log \Psi(Q_k))^{C/4}} \right\}.
$$
By Chebyshev's inequality and \eqref{vari}, we have
$$
\lambda\left( \mathcal{B}_k \right) \ll (\log \Psi(Q_k))^{-C/2} \ll k^{-\sqrt{C}/2}.
$$
Since we assumed that $C>4$, we have $\sum_{k=1}^\infty \lambda(\mathcal{B}_k) < \infty$, and the Borel--Cantelli lemma implies that almost all $\alpha$ are contained in at most finitely many sets $\mathcal{B}_k$. Thus for almost all $\alpha$,
$$
\left| \sum_{q=1}^{Q_k} \mathbbm{1}_{\mathcal{A}_q} (\alpha) - \Psi(Q_k) \right| \leq \frac{\Psi(Q_k)}{(\log \Psi(Q_k))^{C/4}}
$$
holds for all $k \geq k_0(\alpha)$. Clearly, for any $Q \geq 3$ there exists a $k$ such that $Q_k \leq Q < Q_{k+1}$, which also implies that
$$
\sum_{q=1}^{Q_k} \mathbbm{1}_{\mathcal{A}_q} (\alpha)  \leq \sum_{q=1}^{Q} \mathbbm{1}_{\mathcal{A}_q} (\alpha) \leq \sum_{q=1}^{Q_{k+1}} \mathbbm{1}_{\mathcal{A}_q} (\alpha).
$$
Since $\psi \leq 1/2$ by assumption, we have $\Psi(Q_k) \in \left[e^{k^{1/\sqrt{C}}},e^{k^{1/\sqrt{C}}} + 1/2\right]$, and so
$$
\Psi(Q_{k+1}) / \Psi(Q_k) = 1 + O \left(k^{-1+1/\sqrt{C}}\right) = 1 + O \left( \left( \log \Psi(Q_k) \right)^{-\sqrt{C}+1} \right).
$$
From the previous three formulas and the triangle inequality, we deduce that for almost all $\alpha$ there exists a $Q_0 = Q_0(\alpha)$ such that for all $Q \geq Q_0$,
$$
\left| \sum_{q=1}^{Q} \mathbbm{1}_{\mathcal{A}_q} (\alpha) - \Psi(Q) \right| = O \left(  \frac{\Psi(Q)}{(\log \Psi(Q))^{\sqrt{C}-1}} \right).
$$
As $C$ can be chosen arbitrarily large, this proves Theorem~\ref{th1}.

\section{The overlap estimate} \label{sec_3}

In this section, we develop a new estimate for the measure of the overlaps $\mathcal{A}_q \cap \mathcal{A}_r$. For the rest of the paper, let
\begin{equation}\label{D_def}
D(q,r):= \frac{\max \left(  r \psi(q), q \psi(r) \right)}{\gcd (q,r)}, \qquad q,r \in \mathbb{N}.
\end{equation}
The standard bound for the measure of $\mathcal{A}_q \cap \mathcal{A}_r$ is due to Pollington and Vaughan~\cite{pv}: for any $q \neq r$,
\begin{equation} \label{pv*}
\lambda(\mathcal{A}_q \cap \mathcal{A}_r) \ll \lambda(\mathcal{A}_q) \lambda(\mathcal{A}_r) \prod_{\substack{p \mid \frac{qr}{\gcd(q,r)^2}, \\ p>D(q,r)}} \left(1 + \frac{1}{p} \right),
\end{equation}
with an absolute implied constant. Clearly, because of the presence of the implied constant this standard bound cannot be sufficient to deduce Theorem~\ref{th2}. Below we will use a more refined argument from sieve theory which allows us to isolate a main term, and prove an upper bound of the form 
$$
\lambda(\mathcal{A}_q \cap \mathcal{A}_r) \leq \lambda(\mathcal{A}_q) \lambda(\mathcal{A}_r) \left(1 + \textup{[error]}\right),
$$
with an error term that becomes small if there are not too many small primes which divide $q$ and $r$ with different multiplicities (see Lemma~\ref{lemma_over} below for details). 

The following lemma is called the fundamental lemma of sieve theory. We state it in the formulation of~\cite[Theorem 18.11]{kouk}.
\begin{lemma}[Fundamental lemma of sieve theory] \label{lemma_fund}
	Let $(a_n)_{n \geq 1}$ be non-negative reals, such that $\sum_{n=1}^\infty a_n < \infty$. Let $\mathcal{P}$ be a finite set of primes, and write $P = \prod_{p \in \mathcal{P}} p$. Set $y = \max \mathcal{P}$, and $A_d = \sum_{n \equiv 0 \mod d} a_n$. Assume that there exists a multiplicative function $g$ such that $g(p) < p$ for all $p \in \mathcal{P}$, a real number $x$, and positive constants $\kappa,C$ such that
	$$
	A_d =: x \frac{g(d)}{d} + r_d, \qquad d \mid P,
	$$
	and
	$$
	\prod_{p \in (y_1, y_2] \cap \mathcal{P}} \left( 1 - \frac{g(p)}{p} \right)^{-1} < \left( \frac{\log y_2}{\log y_1} \right)^\kappa \left(1 + \frac{C}{\log y_1} \right), \qquad 3/2 \leq y_1 \leq y_2 \leq y.
	$$
	Then, uniformly in $u \geq 1$ we have
	$$
	\sum_{(n,P)=1} a_n = \left( 1 + O ( u^{-u/2} ) \right) x \prod_{p \in \mathcal{P}} \left(1  -\frac{g(p)}{p} \right) + O \left( \sum_{d \leq y^u,~d \mid P} |r_d| \right).
	$$
\end{lemma}

We will also need an estimate for the order of the partial sums of a particular multiplicative function.
\begin{lemma} \label{lemma_mean_value}
	Let $\mathcal{P}$ be a set of odd primes, and define
	$$
	f(n) = \prod_{\substack{p \mid n,\\ p \in \mathcal{P}}} \left(1+ \frac{1}{p-2} \right). 
	$$
	Then for any $x \ge 2$,
	$$
	\sum_{n \leq x} f(n) = x \prod_{p \in \mathcal{P}} \left(1 + \frac{1}{p(p-2)} \right) + O \left(  \log x \right),
	$$
	where the implied constant is absolute.
\end{lemma}

\begin{proof}
	Define $g(n) = \sum_{d \mid n} \mu(d) f(n/d)$, where $\mu$ is the M\"obius function. Note that $f$ and $g$ are multiplicative functions. We have
	\begin{equation} \label{sumfnformula}
	\begin{split} \sum_{n \leq x} f(n) &= \sum_{n \leq x} \sum_{d \mid n} g(d) \\
		&= \sum_{d \leq x} g(d) \left \lfloor \frac{x}{d} \right\rfloor \\
		&= x \sum_{d \leq x} \frac{g(d)}{d} + O \left( \sum_{d \leq x} g(d) \right) \\
		&= x \sum_{d=1}^\infty \frac{g(d)}{d} + O \left(x \sum_{d>x} \frac{g(d)}{d} +  \sum_{d \leq x} g(d) \right). \end{split}
	\end{equation}
	For $p \in \mathcal{P}$,
	\begin{equation*} 
	g(p) = f(p) - 1 = \frac{1}{p-2}, \qquad \text{and} \qquad g(p^m) = 0,~m \geq 2, 
	\end{equation*}
	whereas for $p \not\in \mathcal{P}$, we have $g(p^m)=0$ for all $m \geq 1$. Thus
	$$
	\sum_{d=1}^\infty \frac{g(d)}{d} = \prod_{p \in \mathcal{P}} \left(1 + \frac{1}{p(p-2)} \right),
	$$
	and it remains to estimate the error term in \eqref{sumfnformula}.
	
	Note that $p^m g(p^m) \le p/(p-2) \le 3$ for all prime powers $p^m$. Hence by a general upper bound for the order of partial sums of multiplicative functions (see e.g.\ \cite[Theorem 14.2]{kouk}), the partial sums of $dg(d)$ satisfy
	$$
	\sum_{d \le x} d g(d) \ll x \exp \left( \sum_{p \le x} \frac{pg(p)-1}{p} \right) \ll x \exp \left( \sum_{p>2} \frac{2}{p(p-2)} \right) \ll x.
	$$
	In particular, $\sum_{x \le d \le 2x} g(d)/d \ll x^{-2} \sum_{d \le 2x} d g(d) \ll x^{-1}$, and the first error term in \eqref{sumfnformula} is $x\sum_{d>x} g(d)/d \ll 1$. Further, $\sum_{x \le d \le 2x} g(d) \le x^{-1} \sum_{d \le 2x} d g(d) \ll 1$, and the second error term in \eqref{sumfnformula} is $\sum_{d \le x} g(d) \ll \log x$, as claimed. All implied constants are absolute.
\end{proof}

\begin{lemma}[Overlap estimate] \label{lemma_over}
	For any positive integers $q \neq r$ and any reals $u \ge 1$ and $T \ge 2$, we have
	\begin{equation} \label{overlapclaim}
	\lambda (\mathcal{A}_q \cap \mathcal{A}_r) \le \lambda (\mathcal{A}_q) \lambda (\mathcal{A}_r) \left( 1 + O \left( u^{-u/2} + \frac{T^u \log (D+2) \log T}{D} \right) \right) \prod_{\substack{p \mid \frac{qr}{\gcd(q,r)^2}, \\ p>T}} \left( 1+ \frac{1}{p-1} \right)
	\end{equation}
	with an absolute implied constant, where $D=D(q,r)$ is as in \eqref{D_def}. In particular, for any $C \ge 1$,
	$$
	\lambda (\mathcal{A}_q \cap \mathcal{A}_r) \le \lambda (\mathcal{A}_q) \lambda (\mathcal{A}_r) \left( 1 + O \left( (\log (D+2))^{-C} \right) \right) \prod_{\substack{p \mid \frac{qr}{\gcd(q,r)^2}, \\ p>A}} \left( 1+ \frac{1}{p-1} \right)
	$$
	with an implied constant depending only on $C$, where
	\begin{equation}\label{A_def}
	A=A_C(q,r) := \exp \left( \frac{\log (D+100) \log \log \log (D+100)}{8C \log \log (D+100)} +1 \right) .
	\end{equation}
\end{lemma}

\begin{proof}
	We follow the general strategy of Pollington and Vaughan in~\cite[Section 3]{pv}. If $D<1/2$, then $\psi(q)/q+\psi(r)/r<1/\mathrm{lcm}(q,r)$, hence $\mathcal{A}_q \cap \mathcal{A}_r = \emptyset$, and the claim trivially holds. We may thus assume throughout the rest of the proof that $D \ge 1/2$.
	
	We set
	\begin{eqnarray*}
		\delta = \min \left( \frac{\psi(q)}{q}, \frac{\psi(r)}{r} \right) \qquad \text{and} \qquad \Delta = \max \left( \frac{\psi(q)}{q}, \frac{\psi(r)}{r} \right) ,
	\end{eqnarray*}
	and define the piecewise linear function
	\begin{equation*} 
	w(y) = \left\{ \begin{array}{ll} 2\delta & \text{if $0 \leq y \leq \Delta  - \delta$,} \\ \Delta+\delta-y & \text{if $\Delta - \delta < y \leq \Delta+\delta$,} \\ 0 & \text{otherwise.} \end{array}\right. 
	\end{equation*}
	We can express the measure of $\mathcal{A}_q \cap \mathcal{A}_r$ as
	$$
	\lambda(\mathcal{A}_q \cap \mathcal{A}_r) = \sum_{\substack{1 \leq a \leq q,\\ \gcd(a,q) = 1}}  \sum_{\substack{1 \leq b \leq r,\\ \gcd(b,r) = 1}} w \left(\left| \frac{a}{q} - \frac{b}{r} \right| \right).
	$$
	For any prime $p$, let $u = u(p,q)$ and $v = v(p,r)$ be defined by $q = \prod_p p^u$ and $r = \prod_p p^v$, and let
	$$
	l = \prod_{p:~u=v} p^u, \qquad m = \prod_{p:~u \neq v} p^{\min(u,v)}, \qquad n = \prod_{p:~u \neq v} p^{\max(u,v)}.
	$$
	Following the argument on p.\ 195 of~\cite{pv} (an application of the Chinese remainder theorem, together with a simple counting argument) leads to
	$$
	\sum_{\substack{1 \leq a \leq q,\\ \gcd(a,q) = 1}} \sum_{\substack{1 \leq b \leq r,\\ \gcd(b,r) = 1}} w \left( \left| \frac{a}{q} - \frac{b}{r} \right| \right) = \sum_{\substack{1 \leq c \leq ln,\\ \gcd(c,n) =1}} 2 w\left( \frac{c}{ln} \right) \varphi(m) l \prod_{p \mid \gcd(l,c)} \left(1 - \frac{1}{p} \right) \prod_{\substack{p \mid l, \\p \nmid c}} \left( 1 - \frac{2}{p} \right).
	$$
	
	Assume first that $l$ is odd. By rewriting the right-hand side of the previous formula we see that $\lambda(\mathcal{A}_q \cap \mathcal{A}_r)$ equals
	\[ \begin{split}  & 2 \varphi(m) \frac{\varphi(l)^2}{l} \sum_{\substack{1 \leq c \leq ln,\\ \gcd(c,n) =1}} w\left( \frac{c}{ln} \right)  \prod_{p \mid \gcd(l,c)} \left(1 - \frac{1}{p} \right)^{-1} \prod_{\substack{p \mid l, \\p \nmid c}} \left( \left( 1 - \frac{2}{p} \right) \left( 1 - \frac{1}{p} \right)^{-2} \right) \\ = & 2 \varphi(m) \frac{\varphi(l)^2}{l} \prod_{\substack{p \mid l}}  \left( 1 - \frac{1}{(p-1)^2} \right)  \sum_{\substack{1 \leq c \leq ln,\\ \gcd(c,n) =1}} w\left( \frac{c}{ln} \right) \prod_{\substack {p \mid \gcd(l,c)}} \left(1 + \frac{1}{p-2} \right) . \end{split} \]
	We now find an upper bound for this expression. First, we replace the condition $\gcd(c,n)=1$ by $\gcd(c,n^*)=1$, where $n^*$ denotes the $T$-smooth part of $n$ (i.e.\ $n^*=\prod_{p \le T,~u \neq v}p^{\max (u,v)}$). Next, we fix a large positive integer $K$, and divide $[\Delta-\delta, \Delta+\delta]$ into $K$ subintervals of equal length. Observe that the piecewise constant function
	$$
	w^*(y) = \frac{2 \delta}{K} \left( \left\lfloor \frac{K (\Delta+\delta-y)}{2 \delta} \right\rfloor +1 \right) = \frac{2 \delta}{K} \sum_{k=0}^{K-1} \mathbbm{1}_{[0,\Delta + \delta - 2k \delta /K]}(y)
	$$
	satisfies $w(y) \leq w^*(y)$ for all $y \ge 0$. Therefore $\lambda(\mathcal{A}_q \cap \mathcal{A}_r)$ is bounded above by
	\begin{equation} \label{intersectionupperbound}
	2 \varphi(m) \frac{\varphi(l)^2}{l} \prod_{\substack{p \mid l}}  \left( 1 - \frac{1}{(p-1)^2} \right) \sum_{\substack{1 \leq c \leq ln,\\ \gcd(c,n^*) =1}} w^*\left( \frac{c}{ln} \right) \prod_{\substack {p \mid \gcd(l,c)}} \left(1 + \frac{1}{p-2} \right) .
	\end{equation}
	Here
	$$
	\sum_{\substack{1 \leq c \leq ln,\\ \gcd(c,n^*) =1}} w^* \left( \frac{c}{ln} \right)  \prod_{\substack {p \mid \gcd(l,c)}} \left(1 + \frac{1}{p-2} \right) = \frac{2 \delta}{K} \sum_{k=0}^{K-1} \sum_{\substack{1 \leq c \leq ln (\Delta + \delta - 2k \delta/K),\\ \gcd(c,n^*) =1}} \prod_{\substack {p \mid \gcd(l,c)}} \left(1 + \frac{1}{p-2} \right).
	$$
	Now fix $k \in \{0,\dots,K-1\}$, and set
	$$
	a_c = \prod_{\substack {p \mid \gcd(l,c)}} \left(1 + \frac{1}{p-2} \right), \qquad 1 \leq c \leq ln (\Delta + \delta - 2k \delta/K),
	$$
	and $a_c = 0$ for $c > ln (\Delta + \delta - 2k \delta/K)$. Note that for $d \mid n^*$ we have $a_{dc} = a_c$ as long as $dc \leq ln (\Delta + \delta - 2k \delta/K)$. By Lemma~\ref{lemma_mean_value}, for any $d \mid n^*$ we thus have
	\begin{eqnarray*}
		\sum_{c \equiv 0 \mod d} a_c & = & \sum_{1 \leq c \leq \frac{ln (\Delta + \delta - 2k \delta/K)}{d}} ~\prod_{\substack {p \mid \gcd(l,c)}} \left(1 + \frac{1}{p-2} \right) \\
		& = & \frac{ln (\Delta + \delta - 2k \delta/K)}{d} \prod_{p \mid l} \left(1 + \frac{1}{p(p-2)} \right) + O(\log (D+2)) .
	\end{eqnarray*}
	We have 
	$$
	\sum_{\gcd(c,n^*)=1} a_c = \sum_{\substack{1 \leq c \leq ln (\Delta + \delta - 2k \delta/K),\\ \gcd(c,n^*) =1}} ~\prod_{\substack {p \mid \gcd(l,c)}} \left(1 + \frac{1}{p-2} \right),
	$$
	and by an application of Lemma~\ref{lemma_fund} (with $\mathcal{P}$ the set of prime divisors of $n^*$, $\max \mathcal{P} \le T$ and $|r_d| \ll \log (D+2)$) this is
	$$
	(1 + O (u^{-u/2})) ln \left( \Delta + \delta - \frac{2k \delta}{K} \right) \frac{\varphi(n^*)}{n^*} \prod_{p \mid l} \left(1 + \frac{1}{p(p-2)} \right) + O \left( T^u \log (D+2) \right) . $$
	Since
	$$
	\prod_{\substack{p \mid l}}  \left( 1 - \frac{1}{(p-1)^2} \right)  \prod_{p \mid l} \left(1 + \frac{1}{p(p-2)} \right) = 1,
	$$
	formula \eqref{intersectionupperbound} thus yields that $\lambda(\mathcal{A}_q \cap \mathcal{A}_r)$ is bounded above by 
	$$
    2 \varphi(m) \frac{\varphi(l)^2}{l} \cdot \frac{2 \delta}{K} \sum_{k=0}^{K-1} \left( \left( (1 + O (u^{-u/2})) ln \left( \Delta + \delta - \frac{2k \delta}{K} \right) \frac{\varphi(n^*)}{n^*} + O \left( T^u \log (D+2) \right) \right) \right) .
	$$
	Letting $K \to \infty$, and using $D = \Delta l n$ and $\varphi(n^*)/n^* \ge \prod_{p \le T} (1-1/p) \gg 1/\log T$, we obtain
	\[
	\begin{split} \lambda(\mathcal{A}_q \cap \mathcal{A}_r) &\le 2 \varphi(m) \frac{\varphi(l)^2}{l} 2 \delta \left( (1 + O (u^{-u/2})) ln \Delta \frac{\varphi(n^*)}{n^*} + O \left( T^u \log (D+2) \right) \right) \\ &= 4 \varphi(m) \varphi(l)^2 n \frac{\varphi(n^*)}{n^*} \delta \Delta \left( 1 + O \left( u^{-u/2} +  \frac{T^u \log (D+2) \log T}{D} \right) \right) \\ &= \lambda(\mathcal{A}_q) \lambda (\mathcal{A}_r) \frac{\varphi(n^*)/n^*}{\varphi(n)/n} \left( 1 + O \left( u^{-u/2} +  \frac{T^u \log (D+2) \log T}{D} \right) \right) . \end{split}
	\]
	Finally, observe that
	\[ \frac{\varphi(n^*)/n^*}{\varphi(n)/n} = \frac{1}{\prod_{\substack{p \mid n, \\ p>T}}\left( 1-\frac{1}{p} \right)} = \prod_{\substack{p \mid n, \\ p>T}} \left( 1+ \frac{1}{p-1} \right) . \]
	This establishes \eqref{overlapclaim} for odd $l$.
	
	Assume next that $l$ is even. Then
	\[ \prod_{p \mid \gcd(l,c)} \left(1 - \frac{1}{p} \right) \prod_{\substack{p \mid l, \\p \nmid c}} \left( 1 - \frac{2}{p} \right) = \frac{1}{2} \mathbbm{1}_{\{ 2 \mid c \}} \prod_{\substack{p \mid \gcd(l,c),\\ p>2}} \left(1 - \frac{1}{p} \right) \prod_{\substack{p \mid l, \\p \nmid c,\\p>2}} \left( 1 - \frac{2}{p} \right),
	\]
	and similarly to before we obtain that $\lambda (\mathcal{A}_q \cap \mathcal{A}_r)$ equals
	\[ \begin{split} & 4 \varphi (m) \frac{\varphi(l)^2}{l} \prod_{\substack{p \mid l, \\ p>2}} \left( 1-\frac{1}{p-1} \right) \sum_{\substack{1 \le c \le ln, \\ \gcd(c,n)=1, \\ 2 \mid c}} w \left( \frac{c}{ln} \right) \prod_{\substack{p \mid \gcd(l,c), \\ p>2}} \left( 1+\frac{1}{p-2} \right) \\ = & 4 \varphi (m) \frac{\varphi(l)^2}{l} \prod_{\substack{p \mid l, \\ p>2}} \left( 1-\frac{1}{p-1} \right) \sum_{\substack{1 \le c \le ln/2, \\ \gcd(c,n)=1}} w \left( \frac{c}{ln/2} \right) \prod_{\substack{p \mid \gcd(l,c), \\ p>2}} \left( 1+\frac{1}{p-2} \right) . \end{split}
	\]
	The rest of the proof for odd $l$ applies \emph{mutatis mutandis} to even $l$. This completes the proof of \eqref{overlapclaim}.
	
	Given $C \ge 1$, let us choose
	$$
	u=4C \frac{\log \log (D+100)}{\log \log \log (D+100)} \quad \text{and} \quad T=\exp \left( \frac{\log (D+100) \log \log \log (D+100)}{8C \log \log (D+100)} +1 \right) .
	$$
	One readily checks that $u^{-u/2} \le (\log(D+100))^{-C}$. Using $4/\log \log \log 100 <10$, we also have $T^u \le (D+100)^{1/2} (\log (D+100))^{10C}$, hence
	$$
	\frac{T^u \log (D+2) \log T}{D} \ll \frac{(\log (D+100))^{12C}}{D^{1/2}}
	$$
	is negligible compared to $(\log (D+2))^{-C}$.
\end{proof}	

\section{Second moment bounds} \label{sec_4}

In this section, we show how two second moment bounds, stated as Propositions~\ref{prop_secondmoment1} and~\ref{prop_secondmoment2} below, together with the overlap estimate in Lemma~\ref{lemma_over} imply Theorem~\ref{th2}. These Propositions should be compared to the second moment bound of Koukoulopoulos and Maynard~\cite[Proposition 5.4]{km}, which, together with the overlap estimate of Pollington and Vaughan in equation \eqref{pv*}, implies the Duffin--Schaeffer conjecture.

Let $D(q,r)$ be as in \eqref{D_def}. For the sake of readability, let
\begin{equation}\label{L_def}
L_s (q,r) := \sum_{\substack{p \mid \frac{qr}{\gcd(q,r)^2}, \\ p \geq s}} \frac{1}{p},
\end{equation}
and
\begin{equation}\label{F_def}
F(x)=F_C(x):= \exp \left( \frac{\log (x+100) \log \log \log (x+100)}{8C \log \log (x+100)} +1 \right) .
\end{equation}

\begin{prop}\label{prop_secondmoment1} For any $Q \in \mathbb{N}$ and any real $t \ge 1$, the set
\[ \mathcal{E}_t = \left\{ (q,r) \in [1,Q]^2 \, : \, D(q,r) \le \frac{\Psi(Q)}{t} \right\} \]
satisfies
\[ \sum_{(q,r) \in \mathcal{E}_t} \frac{\varphi(q) \psi(q)}{q} \cdot \frac{\varphi(r) \psi(r)}{r} \ll \frac{\Psi(Q)^2}{t^{1/5}}, \]
with an absolute implied constant.
\end{prop}

\begin{prop}\label{prop_secondmoment2} Let $C \ge 1$ be arbitrary. For any $Q \in \mathbb{N}$ and any real $t \ge 1$, the set
\[ \mathcal{E}_t = \left\{ (q,r) \in [1,Q]^2 \, : \, D(q,r) \le t \Psi(Q) \quad \textrm{and} \quad L_{F(t)}(q,r) \ge \frac{1}{F(t)^{1/4}} \right\} \]
satisfies
\[ \sum_{(q,r) \in \mathcal{E}_t} \frac{\varphi(q) \psi(q)}{q} \cdot \frac{\varphi(r) \psi(r)}{r} \ll \frac{\Psi(Q)^2}{F(t)^{1/2}} \]
with an implied constant depending only on $C$.
\end{prop}

We now present the proof of Theorem~\ref{th2} assuming Propositions~\ref{prop_secondmoment1} and~\ref{prop_secondmoment2}.

\begin{proof}[Proof of Theorem~\ref{th2}] Fix $C>10$, and let $Q \in \mathbb{N}$ be such that $\Psi(Q) \ge 2$. We may assume that $\Psi(Q)$ is large enough in terms of $C$, since otherwise, the claim follows from the estimate \eqref{KMtechnical} of Koukoulopoulos and Maynard.

We partition the index set $[1,Q]^2$ into the sets
\[ \begin{split} \mathcal{E}^{1} &= \left\{ (q,r) \in [1,Q]^2 \, : \, q=r \right\}, \\ \mathcal{E}^{2} &= \left\{ (q,r) \in [1,Q]^2 \, : \, q \neq r, \quad D(q,r) \le \frac{\Psi(Q)}{(\log \Psi (Q))^C}, \quad L_{F(\Psi(Q))}(q,r) \le 1 \right\}, \\ \mathcal{E}^{3} &= \left\{ (q,r) \in [1,Q]^2 \, : \, q \neq r, \quad D(q,r) \le \frac{\Psi(Q)}{(\log \Psi (Q))^C}, \quad L_{F(\Psi(Q))}(q,r) > 1 \right\}, \\ \mathcal{E}^{4} &= \left\{ (q,r) \in [1,Q]^2 \, : \, q \neq r, \quad D(q,r) > \frac{\Psi(Q)}{(\log \Psi (Q))^C}, \quad L_{F(D(q,r))} (q,r) \le \frac{1}{(\log \Psi (Q))^C} \right\}, \\ \mathcal{E}^{5} &= \left\{ (q,r) \in [1,Q]^2 \, : \, q \neq r, \quad D(q,r) > \frac{\Psi(Q)}{(\log \Psi (Q))^C}, \quad L_{F(D(q,r))} (q,r) > \frac{1}{(\log \Psi (Q))^C} \right\}. \end{split} \]
The contribution of $\mathcal{E}^{1}$ is clearly negligible:
\begin{equation}\label{edgeset1sum}
\sum_{(q,r) \in \mathcal{E}^{1}} \lambda (\mathcal{A}_q \cap \mathcal{A}_r) = \sum_{q=1}^Q \lambda (\mathcal{A}_q) = \Psi (Q).
\end{equation}

Now we consider $\mathcal{E}^{2}$. For any $(q,r) \in \mathcal{E}^{2}$, the condition $L_{F(\Psi(Q))}(q,r) \le 1$ together with Mertens' theorem ensures that
\[ \begin{split} \prod_{p \mid \frac{qr}{\gcd (q,r)^2}} \left( 1 + \frac{1}{p-1} \right) &\le \exp \bigg( \sum_{\substack{p \mid \frac{qr}{\gcd (q,r)^2}, \\ p<F (\Psi(Q))}} \frac{2}{p} + \sum_{\substack{p \mid \frac{qr}{\gcd (q,r)^2}, \\ p \ge F (\Psi(Q))}} \frac{2}{p} \bigg) \\ &\ll \exp \left( 2 \log \log F(\Psi (Q)) \right) \\ &\ll (\log \Psi (Q))^2. \end{split} \]
In the last step we used the rough estimate $F(x) \le x$ for large enough $x$. The overlap estimate (Lemma~\ref{lemma_over}) thus shows that for any $(q,r) \in \mathcal{E}^{2}$,
\[ \lambda (\mathcal{A}_q \cap \mathcal{A}_r) \ll \lambda (\mathcal{A}_q) \lambda (\mathcal{A}_r) (\log \Psi (Q))^2 . \]
Applying Proposition~\ref{prop_secondmoment1} with $t=(\log \Psi (Q))^C$ leads to
\begin{equation}\label{edgeset2sum}
\sum_{(q,r) \in \mathcal{E}^{2}} \lambda (\mathcal{A}_q \cap \mathcal{A}_r) \ll \frac{\Psi(Q)^2}{(\log \Psi (Q))^{C/5-2}} .
\end{equation}

Next we consider $\mathcal{E}^{3}$. For any $(q,r) \in \mathcal{E}^{3}$, let $j(q,r)$ be the maximal integer $j$ such that $L_{F(\exp \exp (j))}(q,r) >1$; note that by construction $j(q,r) \ge \lfloor \log \log \Psi (Q) \rfloor$. Let $(q,r) \in \mathcal{E}^{3}$ with $j(q,r)=j$. By definition, $L_{F(\exp \exp (j+1))}(q,r) \le 1$, hence Mertens' theorem implies
\[ \begin{split} \prod_{p \mid \frac{qr}{\gcd (q,r)^2}} \left( 1 + \frac{1}{p-1} \right) &\le \exp \bigg( \sum_{\substack{p \mid \frac{qr}{\gcd (q,r)^2}, \\ p < F(\exp \exp (j+1))}} \frac{2}{p} + \sum_{\substack{p \mid \frac{qr}{\gcd (q,r)^2}, \\ p \ge F(\exp \exp (j+1))}} \frac{2}{p} \bigg) \\ &\ll \exp \left( 2 \log \log F (\exp \exp (j+1)) \right) \\ &\ll \exp (2j) . \end{split} \]
Thus the overlap estimate gives
\[ \lambda (\mathcal{A}_q \cap \mathcal{A}_r) \ll \lambda (\mathcal{A}_q) \lambda (\mathcal{A}_r) \exp (2j) , \]
and applying Proposition~\ref{prop_secondmoment2} with $t=\exp \exp (j)$ leads to
\begin{equation}\label{edgeset3sum}
\begin{split} \sum_{(q,r) \in \mathcal{E}^{3}} \lambda (\mathcal{A}_q \cap \mathcal{A}_r) &= \sum_{j \ge \lfloor \log \log \Psi (Q) \rfloor} \sum_{\substack{(q,r) \in \mathcal{E}^{3}, \\ j(q,r)=j}} \lambda (\mathcal{A}_q \cap \mathcal{A}_r) \\ &\ll \sum_{j \ge \lfloor \log \log \Psi (Q) \rfloor} \exp (2j) \frac{\Psi(Q)^2}{F(\exp \exp (j))^{1/2}} \\ &\ll \frac{\Psi(Q)^2}{(\log \Psi (Q))^C} . \end{split}
\end{equation}
In the last step we used the fact that $F(x)$ increases faster than any power of $\log x$.

Now we consider $\mathcal{E}^{4}$. For any $(q,r) \in \mathcal{E}^{4}$,
\[ \prod_{\substack{p \mid \frac{qr}{\gcd (q,r)^2}, \\ p>F(D(q,r))}} \left( 1+\frac{1}{p-1} \right) \le \exp \left( 2 L_{F(D(q,r))} (q,r) \right) = 1+O \left( \frac{1}{(\log \Psi(Q))^{C}} \right) . \]
The overlap estimate thus gives
\[ \lambda (\mathcal{A}_q \cap \mathcal{A}_r) \le \lambda (\mathcal{A}_q) \lambda (\mathcal{A}_r) \left( 1 + O \left( \frac{1}{(\log \Psi(Q))^{C}} \right) \right) , \]
hence
\begin{equation}\label{edgeset4sum}
\sum_{(q,r) \in \mathcal{E}^{4}} \lambda (\mathcal{A}_q \cap \mathcal{A}_r) \le \Psi (Q)^2 + O \left( \frac{\Psi(Q)^2}{(\log \Psi(Q))^{C}} \right) .
\end{equation}

Finally, we consider $\mathcal{E}^{5}$. For any $(q,r) \in \mathcal{E}^{5}$, let $i(q,r)$ be the maximal integer $i$ such that
\[ L_{F \left( \exp \exp \frac{i}{(\log \Psi(Q))^C} \right)} (q,r) > \frac{1}{2(\log \Psi (Q))^C} .  \]
Note that
\[ L_{F \left( \frac{\Psi(Q)}{(\log \Psi (Q))^C} \right)} (q,r) \ge L_{F(D(q,r))} (q,r) > \frac{1}{(\log \Psi (Q))^C}, \]
therefore
\[ i(q,r) \ge \left\lfloor (\log \Psi (Q))^C \log \log \frac{\Psi(Q)}{(\log \Psi (Q))^C} \right\rfloor . \]
Let $(q,r) \in \mathcal{E}^{5}$ such that $i(q,r)=i$. By definition,
\[ L_{F \left( \exp \exp \frac{i+1}{(\log \Psi(Q))^C} \right)} (q,r) \le \frac{1}{2(\log \Psi (Q))^C}, \]
hence Mertens' theorem shows that
\[ \begin{split} \prod_{p \mid \frac{qr}{\gcd (q,r)^2}} \left( 1+\frac{1}{p-1} \right) &\le \exp \Bigg( \sum_{\substack{p \mid \frac{qr}{\gcd (q,r)^2}, \\ p < F \left( \exp \exp \frac{i+1}{(\log \Psi(Q))^C} \right) }} \frac{2}{p} + \sum_{\substack{p \mid \frac{qr}{\gcd (q,r)^2}, \\ p \ge F \left( \exp \exp \frac{i+1}{(\log \Psi(Q))^C} \right) }} \frac{2}{p} \Bigg) \\ &\ll \exp \left( 2 \log \log F \left( \exp \exp \frac{i+1}{(\log \Psi(Q))^C} \right) \right) \\ &\ll \exp \left( \frac{2i}{(\log \Psi(Q))^C} \right) . \end{split} \]
The overlap estimate thus gives
\[ \lambda (\mathcal{A}_q \cap \mathcal{A}_r) \ll \lambda (\mathcal{A}_q) \lambda (\mathcal{A}_r) \exp \left( \frac{2i}{(\log \Psi(Q))^C} \right) . \]
Another application of Mertens' theorem leads to
\[ \begin{split} \sum_{F \left( \exp \exp \frac{i}{(\log \Psi(Q))^C} \right) \le p \le F \left( \exp \exp \frac{i+1}{(\log \Psi(Q))^C} \right)} \frac{1}{p} = &\log \log F \left( \exp \exp \frac{i+1}{(\log \Psi(Q))^C} \right) \\ &- \log \log F \left( \exp \exp \frac{i}{(\log \Psi(Q))^C} \right) \\ &+ O \left( \exp \left( - \sqrt{\log F \left( \exp \exp \frac{i}{(\log \Psi (Q))^C} \right)} \right) \right) \\ \le & \frac{1}{2(\log \Psi(Q))^C} . \end{split} \]
In the last step we used the facts that $h(x):=\log \log F (\exp \exp (x))$ satisfies $h'(x)=1+o(1)$, and $\log F (\exp \exp (x)) \ge e^{x/2}$ for large enough $x$. It follows that
\[ L_{F \left( \exp \exp \frac{i}{(\log \Psi (Q))^C} \right)} (q,r) \le \frac{1}{2(\log \Psi (Q))^C} + \frac{1}{2(\log \Psi (Q))^C} = \frac{1}{(\log \Psi (Q))^C}, \]
hence $D(q,r) \le \exp \exp \frac{i}{(\log \Psi (Q))^C}$. Applying Proposition~\ref{prop_secondmoment2} with $t=\exp \exp \frac{i}{(\log \Psi (Q))^C}$ thus leads to
\[ \begin{split} \sum_{\substack{(q,r) \in \mathcal{E}^{5}, \\ i(q,r)=i}} \lambda (\mathcal{A}_q \cap \mathcal{A}_r) &\ll \exp \left( \frac{2i}{(\log \Psi(Q))^C} \right) \frac{\Psi(Q)^2}{F \left( \exp \exp \frac{i}{(\log \Psi(Q))^C} \right)^{1/2}} \\ &\ll \frac{\Psi(Q)^2}{\exp \exp \frac{i}{2 (\log \Psi(Q))^C}} , \end{split} \]
and by summing over all possible values of $i$,
\begin{equation}\label{edgeset5sum}
\begin{split} \sum_{(q,r) \in \mathcal{E}^{5}}  \lambda (\mathcal{A}_q \cap \mathcal{A}_r) &\ll \sum_{i \ge \left\lfloor (\log \Psi (Q))^C \log \log \frac{\Psi(Q)}{(\log \Psi (Q))^C} \right\rfloor} \frac{\Psi(Q)^2}{\exp \exp \frac{i}{2 (\log \Psi(Q))^C}} \\ &\ll \sum_{m \ge \log \log \frac{\Psi (Q)}{(\log \Psi(Q))^C}} \frac{\Psi(Q)^2 (\log \Psi(Q))^C}{\exp \exp \frac{m}{2}} \\ &\ll \frac{\Psi(Q)^2}{(\log \Psi (Q))^C} . \end{split}
\end{equation}
Combining formulas \eqref{edgeset1sum}--\eqref{edgeset5sum} shows that
\[ \sum_{q,r=1}^Q \lambda (\mathcal{A}_q \cap \mathcal{A}_r) \le \Psi (Q)^2 + O \left( \frac{\Psi (Q)^2}{(\log \Psi (Q))^{C/5-2}} \right), \]
as claimed.
\end{proof}

\section{GCD graphs: notations and basic properties} \label{sec_5}

The proof of the Duffin--Schaeffer conjecture given by Koukoulopoulos and Maynard in~\cite{km} is based on a concept called ``GCD graphs'', which they introduced in that paper. Very roughly speaking, a GCD graph encodes information on the divisor structure of a set of integers. To each GCD graph, a ``quality'' can be assigned, and the key argument in~\cite{km} is that one can iteratively pass to subgraphs of the original GCD graph in such a way that in each step either the quality increases and/or the divisor structure becomes more regular. At the end of this procedure, one has a graph that either has particularly high quality, or a very regular divisor structure. High quality directly implies that the density of the edge set, essentially controlling the influence of the bad pairs $(q,r)$ in such sets as $\mathcal{E}^{1}$ -- $\mathcal{E}^{5}$ of the previous section, is small, leading to the desired result. If one cannot achieve high quality, then one obtains a GCD subgraph that has perfect control of the divisor structure of the underlying set of integers; in this case, results on the ``anatomy of integers'' can be used to show that the problematic factor $\prod_{\substack{p \mid \frac{qr}{\gcd(q,r)^2}}} \left(1 + \frac{1}{p} \right)$ in the overlap estimate can only be large for a very small proportion of pairs $(q,r)$, again leading to the desired result.

We do not give a fully detailed presentation of the notion of a GCD graph here, and refer the reader to Section 6 of~\cite{km} instead. However, for the convenience of the reader, we will recall the basic definitions and some of the basic properties of GCD graphs.

A GCD graph is a septuple $G = (\mu,\mathcal{V},\mathcal{W},\mathcal{E},\mathcal{P},f,g)$, for which the following properties hold.
\begin{enumerate}[a)]
\item $\mu$ is a measure on $\mathbb{N}$ for which $\mu(n)<\infty$ for all $n$. This measure is extended to $\mathbb{N}^2$ by defining
$$
\mu(\mathcal{N}) = \sum_{(n_1,n_2) \in \mathcal{N}} \mu(n_1) \mu(n_2), \qquad \mathcal{N} \subseteq \mathbb{N}^2.
$$
\item The \emph{vertex sets} $\mathcal{V}$ and $\mathcal{W}$ are finite sets of positive integers.
\item The \emph{edge set} $\mathcal{E}$ is a subset of $\mathcal{V} \times \mathcal{W}$.
\item $\mathcal{P}$ is a set of primes.
\item $f$ and $g$ are functions from $\mathcal{P}$ to $\mathbb{Z}_{\geq 0}$ such that for all $p \in \mathcal{P}$,
\begin{enumerate}[(i)]
\item $p^{f(p)} \mid v$ for all $v \in \mathcal{V}$ and $p^{g(p)} \mid w$ for all $w \in \mathcal{W}$;
\item if $(v,w) \in \mathcal{E}$, then $p^{\min(f(p),g(p))} \parallel \gcd(v,w)$;
\item if $f(p) \neq g(p)$, then $p^{f(p)} \parallel v$ for all $v \in \mathcal{V}$ and $p^{g(p)} \parallel w$ for all $w \in \mathcal{W}$.
\end{enumerate}
\end{enumerate}

For two GCD graphs $G = (\mu,\mathcal{V},\mathcal{W},\mathcal{E},\mathcal{P},f,g)$ and $G' = (\mu',\mathcal{V}',\mathcal{W}',\mathcal{E}',\mathcal{P}',f',g')$ we say that $G'$ is a GCD subgraph of $G$, and write $G' \preceq G$, if
$$
\mu' = \mu, \quad \mathcal{V}' \subseteq \mathcal{V}, \quad  \mathcal{W}' \subseteq \mathcal{W}, \quad \mathcal{E}' \subseteq \mathcal{E}, \quad \mathcal{P}' \supseteq \mathcal{P},
$$
and if $f$ resp.\ $g$ coincide with $f'$ resp.\ $g'$ on $\mathcal{P}$. 

For given $\mathcal{V}$ and $k \geq 0$ we define $\mathcal{V}_{p^k} = \{v \in \mathcal{V}:~p^k \parallel v\}$. We write $\mathcal{E}_{p^k,p^\ell} = \mathcal{E} \cap (\mathcal{V}_{p^k} \times \mathcal{W}_{p^\ell})$. It turns out that for $p \not \in \mathcal{P}$, the GCD graph
$$
G_{p^k,p^\ell} := (\mu, \mathcal{V}_{p^k},\mathcal{W}_{p^\ell}, \mathcal{E}_{p^k,p^\ell}, \mathcal{P} \cup \{p\}, f_{p^k},g_{p^\ell})
$$
is a GCD subgraph of $G$ (where $f_{p^k}$ resp.\ $g_{p^\ell}$ are defined in such a way that they coincide with $f$ resp.\ $g$ on $\mathcal{P}$, and $f_{p^k}(p)=k$ and $g_{p^{\ell}}(p)=\ell$).

For a GCD graph $G = (\mu,\mathcal{V},\mathcal{W},\mathcal{E},\mathcal{P},f,g)$ we define 
\begin{enumerate}[(i)]
\item The \emph{edge density} 
$$
\delta(G) = \frac{\mu(\mathcal{E})}{\mu(\mathcal{V}) \mu(\mathcal{W})},
$$
provided that $\mu(\mathcal{V}) \mu(\mathcal{W}) \neq 0$.
If $\mu(\mathcal{V}) \mu(\mathcal{W}) = 0$, we define $\delta(G)$ to be $0$.
\item The \emph{neighborhood sets} 
$$
\Gamma_G(v) = \left\{ w \in \mathcal{W}:~(v,w) \in \mathcal{E} \right\}, \qquad v \in \mathcal{V},
$$
and
$$
\Gamma_G(w) = \left\{v \in \mathcal{V}:~(v,w) \in \mathcal{E} \right\}, \qquad w \in \mathcal{W}.
$$
\item The set $\mathcal{R}(G)$ of primes that have not (yet) been accounted for in the GCD graph:
$$
\mathcal{R}(G) = \left\{ p \not\in \mathcal{P}:~\exists (v,w) \in \mathcal{E} \text{ such that } p \mid \gcd(v,w) \right\}.
$$
\item The \emph{quality}
$$
q(G) = \delta (G)^{10}  \mu(\mathcal{V}) \mu(\mathcal{W}) \prod_{p \in \mathcal{P}} \frac{p^{|f(p)-g(p)|}}{\left(1 - \mathbbm{1}_{f(p)=g(p) \geq 1}/p\right)^2 \left(1 - p^{-31/30} \right)^{10}}.
$$
\end{enumerate}
This notion of quality of a GCD graph is an ad-hoc definition, which turns out to serve the required purpose for the argument of~\cite{km}. We refer to~\cite{km} for the heuristic reasoning which led to this particular definition. It is possible that a modified notion of quality would be better suited for the argument in the present paper. However, we preferred to stick to the original definition of quality from~\cite{km}, since this allows us to directly use a large part of the iteration procedure from~\cite{km} without the need to adapt it to a modified framework. 

We also introduce
\[ \mathcal{R}^{\twonotes}(G) := \left\{ p \in \mathcal{R}(G) \, : \, \forall k \ge 0 \,\,\, \min \left\{ \frac{\mu(\mathcal{V}_{p^k})}{\mu(\mathcal{V})},\frac{\mu(\mathcal{W}_{p^k})}{\mu(\mathcal{W})} \right\} \le 1 - \frac{1}{\sqrt{p}} \right\} . \]
This should be compared to the sets $\mathcal{R}^{\sharp}(G)$ and $\mathcal{R}^{\flat}(G)$ used in~\cite{km}, the latter of which is defined analogous to our $\mathcal{R}^{\twonotes}(G)$ but with $1-10^{40}/p$ instead of $1-1/\sqrt{p}$. Finally, we define
\[\mathcal{P}_{\text{diff}}(G) := \{p \in \mathcal{P} \, : \, f(p) \neq g(p) \} .\]

Among the basic properties of GCD graphs are the facts that $G_1 \preceq G_2$ and $G_2 \preceq G_3$ together imply $G_1 \preceq G_3$ (transitivity), and that $G_1 \preceq G_2$ implies $\mathcal{R}(G_1) \subseteq \mathcal{R}(G_2)$. However, in general $G_1 \preceq G_2$ does not imply that $\mathcal{R}^{\twonotes}(G_1) \subseteq \mathcal{R}^{\twonotes}(G_2)$.

\section{Good GCD subgraphs} \label{sec_6}

In this section, we state two results on the existence of a ``good'' GCD subgraph of an arbitrary GCD graph with trivial multiplicative data (i.e.\ $\mathcal{P}=\emptyset$) in the form of Propositions~\ref{prop_goodgcd1} and~\ref{prop_goodgcd2} below; these should be compared to~\cite[Proposition 7.1]{km}. We then show how Proposition~\ref{prop_secondmoment1} resp.\ \ref{prop_secondmoment2} follow from Proposition~\ref{prop_goodgcd1} resp.\ \ref{prop_goodgcd2}.
 
\begin{prop}\label{prop_goodgcd1}
Let $G=(\mu,\mathcal{V},\mathcal{W},\mathcal{E},\emptyset,f_\emptyset,g_\emptyset)$ be a GCD graph with trivial set of primes and edge density $\delta (G)>0$. Then there exists a GCD subgraph $G' = (\mu,\mathcal{V}',\mathcal{W}',\mathcal{E}',\mathcal{P}',f',g')$ of $G$ such that
\begin{enumerate}
 \item $\mathcal{R} (G') = \emptyset$.
 \item For all $v \in \mathcal{V}'$, we have $\mu(\Gamma_{G'}(v)) \geq \frac{9 \delta(G')}{10} \mu(\mathcal{W}')$.
 \item For all $w \in \mathcal{W}'$, we have $\mu(\Gamma_{G'}(w)) \geq \frac{9 \delta(G')}{10} \mu(\mathcal{V}')$.
 \item $q(G') \gg q(G)$ with an absolute implied constant.
\end{enumerate}
\end{prop}

\begin{prop}\label{prop_goodgcd2}
Let $G= (\mu,\mathcal{V},\mathcal{W},\mathcal{E},\emptyset,f_{\emptyset},g_{\emptyset})$ be a GCD graph with trivial set of primes, and let $C \ge 1$. Assume that
\[ \mathcal{E} \subseteq \left\{(v,w) \in \mathcal{V} \times \mathcal{W}: L_{F(t)}(v,w) \geq \frac{1}{F(t)^{1/4}}\right\} \quad \textrm{and} \quad \delta(G)  \ge \frac{1}{F(t)^{1/2}}\]
with some $t \ge 1$ sufficiently large in terms of $C$. Then there exists a GCD subgraph $G' = (\mu,\mathcal{V}',\mathcal{W}',\mathcal{E}',\mathcal{P}',f',g')$ of $G$ such that
	\begin{enumerate}
		\item $\mathcal{R}(G') = \emptyset$.
		\item For all $v \in \mathcal{V}'$, we have $\mu(\Gamma_{G'}(v)) \geq \frac{9\delta(G')}{10}\mu(\mathcal{W}')$.
		\item For all $w \in \mathcal{W}'$, we have $\mu(\Gamma_{G'}(w)) \geq \frac{9\delta(G')}{10}\mu(\mathcal{V}')$.
		\item One of the following holds:
		\begin{enumerate}
			\item[(i)] $q(G') \gg t^3 q(G)$ with an implied constant depending only on $C$.
			\item[(ii)] $q(G') \gg q(G)$ with an implied constant depending only on $C$, and for any $(v,w) \in \mathcal{E}'$, if we write $v = v'\prod_{p \in \mathcal{P}'}p^{f'(p)}$ and $w = w'\prod_{p \in \mathcal{P}'}p^{g'(p)}$, then $L_{F(t)}(v',w') \geq \frac{1}{2F(t)^{1/4}}$.
		\end{enumerate}
	\end{enumerate}
\end{prop}

\begin{proof}[Proof of Proposition~\ref{prop_secondmoment1}]
Let $\psi: \mathbb{N} \to [0,1/2]$ be a function, let $Q \in \mathbb{N}$ and let $t \ge 1$. Consider the GCD graph $G=(\mu, \mathcal{V}, \mathcal{W}, \mathcal{E}, \emptyset, f_{\emptyset}, g_{\emptyset})$ with the measure $\mu(v)=\frac{\varphi(v)\psi(v)}{v}$, the vertex sets $\mathcal{V}=\mathcal{W}=[1,Q]^2$, and the edge set
\[ \mathcal{E}= \left\{ (v,w) \in [1,Q]^2 \, : \, D(v,w) \le \frac{\Psi(Q)}{t} \right\} . \]
Note that $\mu(\mathcal{V})=\mu(\mathcal{W})=\Psi(Q)/2$. In the language of GCD graphs, the claim of Proposition \ref{prop_secondmoment1} can equivalently be written as $\mu (\mathcal{E}) \ll \Psi(Q)^2 /t^{1/5}$, that is, $\delta(G) \ll t^{-1/5}$.

By Proposition~\ref{prop_goodgcd1}, there exists a GCD subgraph $G'=(\mu, \mathcal{V}', \mathcal{W}', \mathcal{E}', \mathcal{P}', f', g')$ of $G$ having properties a)--d) of the proposition. Following the steps in~\cite[Proof of Proposition 6.3 assuming Proposition 7.1]{km}, from properties a)--c) we deduce $q(G') \ll \Psi(Q)^2/t^2$. Since $G$ has trivial set of primes, by the definition of quality and property d),
\[ \delta(G)^{10} \mu (\mathcal{V}) \mu (\mathcal{W}) = q(G) \ll q(G') \ll \frac{\Psi(Q)^2}{t^2} . \]
Therefore $\delta(G) \ll t^{-1/5}$, as claimed.
\end{proof}

For the proof of Proposition~\ref{prop_secondmoment2} we will need the following fact about the ``anatomy of integers''; compare this to~\cite[Lemma 7.3]{km}, which is a similar result for a fixed value of $c$ on the right-hand side, rather than allowing $c \to 0$ as in view of Lemma~\ref{lemma_over} above will be necessary for our application.
\begin{lemma}\label{lemma_anatomy}
For any real $x,t \geq 1$ and $0<c \le 1$,
$$
\bigg| \bigg\{n \leq x \, : \, \sum_{\substack{p \mid n,\\ p \geq t}} \frac{1}{p} \geq c \bigg\} \bigg| \ll xe^{- 100 ct}
$$
with an absolute implied constant.
\end{lemma}

\begin{proof}
An application of the Markov inequality gives
\[ \begin{split} \bigg| \bigg\{ n \le x \, : \, \sum_{\substack{p \mid n, \\ p \ge t}} \frac{1}{p} \ge c \bigg\} \bigg| &= \bigg| \bigg\{ n \le x \, : \, \exp \bigg( 100t \sum_{\substack{p \mid n, \\ p \ge t}} \frac{1}{p} \bigg) \ge \exp \left( 100ct \right) \bigg\} \bigg| \\ &\le e^{- 100ct} \sum_{n \le x} \prod_{\substack{p \mid n, \\ p \ge t}} e^{100t/p}
.\end{split} \]
Now let $f$ be the multiplicative function defined at prime powers as $f(p^m) = e^{100 t/p}$ if $p \ge t$, and $f(p^m)=1$ if $p<t$. Note that $f(p^m) \le e^{100}$ at all prime powers. Hence by~\cite[Theorem 14.2]{kouk} the partial sums of $f$ satisfy
\[ \begin{split} \sum_{n \le x} \prod_{\substack{p \mid n, \\ p \ge t}} e^{100t/p} = \sum_{n \le x} f(n) \ll x \exp \left( \sum_{p \le x} \frac{f(p)-1}{p} \right) &= x \exp \left( \sum_{t \le p \le x} \frac{e^{100t/p}-1}{p} \right) \\ &= x \exp \left( O \left( \sum_{p \ge t} \frac{t}{p^2} \right) \right) \\ &\ll x, \end{split} \]
where the implied constants are absolute.
\end{proof}

\begin{proof}[Proof of Proposition~\ref{prop_secondmoment2}]
Let $\psi: \mathbb{N} \to [0,1/2]$ be a function, let $Q \in \mathbb{N}$ and let $t \ge 1$. Consider the GCD graph $G=(\mu, \mathcal{V}, \mathcal{W}, \mathcal{E}, \emptyset, f_{\emptyset}, g_{\emptyset})$ with the measure $\mu(v)=\frac{\varphi(v)\psi(v)}{v}$, the vertex sets $\mathcal{V}=\mathcal{W}=[1,Q]^2$, and the edge set
\[ \mathcal{E}= \left\{ (v,w) \in [1,Q]^2 \, : \, D(v,w) \le t \Psi(Q) \quad \textrm{and} \quad L_{F(t)}(v,w) \ge \frac{1}{F(t)^{1/4}} \right\} . \]
Note that $\mu(\mathcal{V})=\mu(\mathcal{W})=\Psi(Q)/2$. In the language of GCD graphs, the claim can equivalently be written as $\mu (\mathcal{E}) \ll \Psi(Q)^2/F(t)^{1/2}$, that is, $\delta(G) \ll F(t)^{-1/2}$. We may assume in the sequel that $\delta(G) \ge F(t)^{-1/2}$ and that $t$ and $F(t)$ are large enough in terms of $C$, since otherwise the claim trivially holds.

By Proposition~\ref{prop_goodgcd2}, there exists a GCD subgraph $G'=(\mu, \mathcal{V}', \mathcal{W}', \mathcal{E}', \mathcal{P}', f', g')$ of $G$ having properties a)--d) of the proposition. Let $a=\prod_{p \in \mathcal{P}'}p^{f'(p)}$ and $b=\prod_{p \in \mathcal{P}'}p^{g'(p)}$. By the definition of a GCD graph, $a \mid v$ for all $v \in \mathcal{V}'$ and $b \mid w$ for all $w \in \mathcal{W}'$. Since $\mathcal{R}(G')=\emptyset$, we also have $\gcd (v,w)=\gcd (a,b)$ for all $(v,w) \in \mathcal{E}'$. Following the steps in~\cite[Proof of Proposition 6.3 assuming Proposition 7.1]{km}, we deduce from properties a)--c) of Proposition~\ref{prop_goodgcd2} that
\begin{equation}\label{qG'upperboundfromkm}
q(G') \ll ab \Psi(Q)^2 t^2 \sum_{(v,w) \in \mathcal{E}'} \frac{1}{w_0 v_{\max}(w)} \le \Psi (Q)^2 t^2 ,
\end{equation}
where $w_0=\max \mathcal{W}'$ and $v_{\max}(w)=\max \{ v \in \mathcal{V}' \, : \, (v,w) \in \mathcal{E}' \}$.

Assume first that $G'$ satisfies property d)(i) in Proposition~\ref{prop_goodgcd2}, that is, $q(G') \gg t^3 q(G)$. Since $G$ has trivial set of primes, by the definition of quality and \eqref{qG'upperboundfromkm} we obtain
\[ \delta(G)^{10} \mu (\mathcal{V}) \mu (\mathcal{W}) = q(G) \ll t^{-3} q(G') \ll \frac{\Psi(Q)^2}{t}. \]
Therefore $\delta(G) \ll t^{-1/10} \ll F(t)^{-1/2}$, as claimed.

Assume next, that $G'$ satisfies property d)(ii) in Proposition~\ref{prop_goodgcd2}, that is, $q(G') \gg q(G)$, and for any $(v,w) \in \mathcal{E}'$, if we write $v = av'$ and $w = bw'$, then $L_{F(t)}(v',w') \geq \frac{1}{2F(t)^{1/4}}$. Note that here $\gcd (v',w')=1$. As in the first case, we have
\[ \begin{split} \delta(G)^{10} \mu (\mathcal{V}) \mu (\mathcal{W}) = q(G) \ll q(G') &\ll ab \Psi(Q)^2 t^2 \sum_{(v,w) \in \mathcal{E}'} \frac{1}{w_0 v_{\max}(w)} \\ &\le \frac{ab \Psi(Q)^2 t^2}{w_0} \sum_{1 \le w' \le w_0/b} \frac{1}{v_{\max}(bw')} \sum_{\substack{1 \le v' \le v_{\max}(bw')/a, \\ L_{F(t)}(v',w') \ge 1/(2F(t)^{1/4})}} 1 . \end{split} \]
For the sake of readability, define $R_s(n)=\sum_{p \mid n,~p \ge s} 1/p$ for any $n \in \mathbb{N}$ and $s \ge 1$. Then $1/(2F(t)^{1/4}) \le L_{F(t)}(v',w') = R_{F(t)}(v') + R_{F(t)}(w')$ implies that $R_{F(t)}(v') \ge 1/(4F(t)^{1/4})$ or $R_{F(t)}(w') \ge 1/(4F(t)^{1/4})$. The previous formula thus shows that $\delta (G)^{10} \ll S_1+S_2$ with
\[ \begin{split} S_1 &= \frac{ab t^2}{w_0} \sum_{1 \le w' \le w_0/b} \frac{1}{v_{\max}(bw')} \sum_{\substack{1 \le v' \le v_{\max}(bw')/a, \\ R_{F(t)}(v') \ge 1/(4F(t)^{1/4})}} 1, \\ S_2 &= \frac{ab t^2}{w_0} \sum_{\substack{1 \le w' \le w_0/b, \\ R_{F(t)}(w') \ge 1/(4F(t)^{1/4})}} \frac{1}{v_{\max}(bw')} \sum_{1 \le v' \le v_{\max}(bw')/a} 1 . \end{split} \]
An application of Lemma~\ref{lemma_anatomy} with $x=v_{\max}(bw')/a$ and $c=1/(4F(t)^{1/4})$ yields
\[ S_1 \ll \frac{bt^2}{w_0} \sum_{1 \le w' \le w_0/b} \exp \left( -25 F(t)^{3/4} \right) = t^2 \exp \left( -25 F(t)^{3/4} \right) \ll t^{-100} . \]
Another application of Lemma~\ref{lemma_anatomy} with $x=w_0/b$ and $c=1/(4F(t)^{1/4})$ similarly yields
\[ S_2 = \frac{bt^2}{w_0} \sum_{\substack{1 \le w' \le w_0/b, \\ R_{F(t)}(w') \ge 1/(4F(t)^{1/4})}} 1 \ll t^2 \exp \left( -25 F(t)^{3/4} \right) \ll t^{-100}. \]
Therefore $\delta(G) \ll (S_1+S_2)^{1/10} \ll t^{-10} \ll F(t)^{-10}$, as claimed.
\end{proof}

\section{Four technical lemmas} \label{sec_7}

In this section, we state four lemmas on GCD subgraphs, and show that Propositions~\ref{prop_goodgcd1} and~\ref{prop_goodgcd2} follow from these four lemmas. The key technical improvement in comparison with the iteration argument of~\cite{km} is in Lemma~\ref{quality_density_lemma} below, which more carefully balances the quality gain versus the potential density loss of the iteration procedure. The ratio of quality gain vs.\ density loss which is necessary for the proof of Theorem~\ref{th2} is determined by the range of admissible parameters $u$ and $A$ in Lemma~\ref{lemma_over}, and what Lemma~\ref{quality_density_lemma} provides is just enough for a successful completion of the proof. Lemma~\ref{lem84}, which should be compared to~\cite[Lemma 8.4]{km}, and Lemma~\ref{empty_R} follow from results in~\cite{km} in a more or less straightforward way. Finally, for the convenience of the reader, we cite~\cite[Lemma 8.5]{km} in the form of Lemma~\ref{lem85}.

\begin{lemma}\label{quality_density_lemma}
	Let $G = (\mu,\mathcal{V},\mathcal{W},\mathcal{E},\emptyset,f_{\emptyset},g_{\emptyset})$ be a GCD graph with trivial set of primes and $\delta(G) > 0$. Let $C \ge 1$, and let $t \ge 1$ be sufficiently large in terms of $C$. Then there exists a GCD subgraph $G' \preceq G$ such that $R^{\twonotes}(G') = \emptyset$, and at least one of the following two statements holds:
	\begin{enumerate}
		\item $q(G') \ge t^3 q(G)$.
		\item $q(G') \gg q(G), \quad \frac{\delta(G')}{\delta(G)} \ge \frac{1}{F(t)^{1/4}} ,\quad \lvert \mathcal{P}_{\text{diff}}(G')\rvert \le \log t$ with an implied constant depending only on $C$.
	\end{enumerate}
\end{lemma}

\begin{lemma}\label{lem84}
	Let $G= (\mu,\mathcal{V},\mathcal{W},\mathcal{E},\mathcal{P},f,g)$ be a GCD graph. Assume that
	\[\delta(G) \geq \frac{1}{s^{1/4}},\quad \mathcal{R}^{\twonotes}(G) = \emptyset, \quad 
	\mathcal{E} \subseteq \left\{(v,w) \in \mathcal{V} \times \mathcal{W}: L_{s}(v,w) \geq \frac{1}{s^{1/4}}\right\} \]
	with a sufficiently large $s \ge 1$. Then there exists a GCD subgraph $G' = (\mu,\mathcal{V},\mathcal{W},\mathcal{E}',\mathcal{P},f,g)$ of $G$ such that
	\[ q(G') \geq \frac{q(G)}{2} \quad \text{and} \quad \mathcal{E}' \subseteq \Bigg\{(v,w) \in \mathcal{V} \times \mathcal{W}: \sum_{\substack{p \mid \frac{vw}{\gcd(v,w)^2},\\ p \geq s, \,\, p \notin \mathcal{R}(G)}} \frac{1}{p} \geq \frac{3}{4s^{1/4}}\Bigg\}.\]
\end{lemma}

\begin{lemma}\label{empty_R}
	Let $G=(\mu,\mathcal{V},\mathcal{W},\mathcal{E},\mathcal{P},f,g)$ be a GCD graph with $\delta(G) > 0$. Then there exists a GCD subgraph $G'=(\mu, \mathcal{V}', \mathcal{W}', \mathcal{E}', \mathcal{P}', f', g')$ of $G$ such that 
	\[ \mathcal{P}' \subseteq \mathcal{P} \cup \mathcal{R}(G), \quad \mathcal{R}(G') = \emptyset, \quad q(G') \gg q(G) \]
	with an absolute implied constant.
\end{lemma}

\begin{lemma}[{\cite[Lemma 8.5]{km}}] \label{lem85}
	Let $G= (\mu,\mathcal{V},\mathcal{W},\mathcal{E},\mathcal{P},f,g)$ be a GCD graph with $\delta(G) > 0$. Then there exists a GCD subgraph $G' = (\mu,\mathcal{V},\mathcal{W},\mathcal{E}',\mathcal{P},f,g)$ of $G$ such that:
	\begin{enumerate}
		\item $q(G') \geq q(G)$.
		\item $\delta(G') \geq \delta(G)$.
		\item For all $v \in \mathcal{V}'$ and $w \in \mathcal{W}'$, we have
		\[\mu(\Gamma_{G'}(v)) \geq \frac{9\delta(G')}{10}\mu(\mathcal{W}') \quad \textrm{and} \quad \mu(\Gamma_{G'}(w)) \geq \frac{9\delta(G')}{10}\mu(\mathcal{V}').\]
	\end{enumerate}
\end{lemma}

We now show how Lemmas~\ref{quality_density_lemma}--\ref{lem85} imply Propositions~\ref{prop_goodgcd1} and~\ref{prop_goodgcd2}.

\begin{proof}[Proof of Proposition~\ref{prop_goodgcd1}] Apply Lemma~\ref{empty_R} to $G$ to obtain a GCD subgraph $G^{(1)} \preceq G$ with $\mathcal{R}(G^{(1)})= \emptyset$ and $q(G^{(1)}) \gg q(G)$, satisfying properties a) and d). Next, apply Lemma~\ref{lem85} to $G^{(1)}$ to obtain a GCD subgraph $G^{(2)} \preceq G^{(1)}$ which additionally satisfies properties b) and c).
\end{proof}

\begin{proof}[Proof of Proposition~\ref{prop_goodgcd2}] We follow~\cite[Proof of Proposition 7.1]{km}, although the ordering of the different stages needs to be changed. It suffices to prove the existence of a GCD subgraph which satisfies properties a) and d). Indeed, applying Lemma~\ref{lem85} to such a subgraph, we obtain a GCD subgraph that satisfies all required properties a)--d).

We start by applying Lemma~\ref{quality_density_lemma} to $G$, and obtain a GCD subgraph $G^{(1)} \preceq G$ such that $\mathcal{R}^{\twonotes}(G^{(1)}) = \emptyset$, and $G^{(1)}$ satisfies at least one of the following properties:
	\begin{enumerate}[A)]
		\item $q(G^{(1)}) \ge t^3 q(G)$.
		\item $q(G^{(1)}) \gg q(G), \quad \frac{\delta(G^{(1)})}{\delta(G)} \ge \frac{1}{F(t)^{1/4}},\quad \lvert\mathcal{P}_{\text{diff}}(G^{(1)})\rvert \le \log t$.
	\end{enumerate}
	We distinguish between two cases depending on whether A) or B) is satisfied.

	Case A). Assume that $q(G^{(1)}) \ge t^3 q(G)$. We apply Lemma~\ref{empty_R} to obtain a GCD subgraph $G^{(2A)} \preceq G^{(1)}$ with $\mathcal{R}(G^{(2A)}) = \emptyset$ and $q(G^{(2A)}) \gg q(G^{(1)})$. Then $G^{(2A)}$ satisfies properties a) and d)(i) in Proposition~\ref{prop_goodgcd2}. This finishes the proof for Case A).
	
	Case B). Assume that $q(G^{(1)}) \gg q(G),\;\frac{\delta(G^{(1)})}{\delta(G)} \ge \frac{1}{F(t)^{1/4}},\; \lvert\mathcal{P}_{\text{diff}}(G^{(1)})\rvert \le \log t$. First, we remove the effect of the large primes in $\mathcal{R}(G^{(1)})$ on $L_{F(t)}(v,w)$. By the assumption $\delta(G) \ge 1/ F(t)^{1/2}$, we have $\delta(G^{(1)}) \ge 1/F(t)^{1/4}$. We can thus apply Lemma~\ref{lem84} to $G^{(1)}$ with $s = F(t)$ to obtain a GCD subgraph $G^{(2B)} \preceq G^{(1)}$ with edge set $\mathcal{E}^{(2B)}$ such that
	\[ q(G^{(2B)}) \geq \frac{q(G^{(1)})}{2} \quad \textrm{and} \quad \mathcal{E}^{(2B)} \subseteq \Bigg\{(v,w) \in \mathcal{V} \times \mathcal{W}: \sum_{\substack{p \mid \frac{vw}{\gcd(v,w)^2},\\ p \geq F(t), \,\, p \notin \mathcal{R}(G^{(1)})}} \frac{1}{p} \geq \frac{3}{4F(t)^{1/4}}\Bigg\}.\]
	Now we remove the contribution of the primes in $\mathcal{P}_{\text{diff}}(G^{(1)})$. Using $\lvert\mathcal{P}_{\text{diff}}(G^{(1)})\rvert \le \log t$, we obtain that for any $(v,w) \in \mathcal{E}^{(2B)}$,
	\[ \sum_{\substack{p \mid \frac{vw}{\gcd(v,w)^2},\\ p \geq F(t), \,\, p \in \mathcal{P}_{\text{diff}}(G^{(1)})}} \frac{1}{p} \le \frac{\log t}{F(t)} \le \frac{1}{4F(t)^{1/4}} \]
	for large enough $t$. Hence for any $(v,w) \in \mathcal{E}^{(2B)}$,
	\begin{equation}\label{without_diff}\sum_{\substack{p \mid \frac{vw}{\gcd(v,w)^2},\\ p \geq F(t), \,\, p \notin \mathcal{R}(G^{(1)}) \cup \mathcal{P}_{\text{diff}}(G^{(1)})}} \frac{1}{p} \geq \frac{1}{2 F(t)^{1/4}}.
	\end{equation}
	Finally, we apply Lemma~\ref{empty_R} to $G^{(2B)}$ to obtain a GCD subgraph $G^{(3B)} \preceq G^{(2B)}$ such that
	\[ \mathcal{R}(G^{(3B)}) = \emptyset \quad \textrm{and} \quad q(G^{(3B)}) \gg q(G^{(2B)}) \gg q(G). \]
	Thus $G^{(3B)}$ satisfies property a) in Proposition~\ref{prop_goodgcd2}. Following the steps in Stage 4b of~\cite[Proof of Proposition 7.1]{km}, we deduce from \eqref{without_diff} that $G^{(3B)}$ satisfies property d)(ii) as well. This finishes the proof for Case B).
\end{proof}

\section{Proof of Lemmas~\ref{lem84} and~\ref{empty_R}} \label{sec_8}
\begin{proof}[Proof of Lemma~\ref{lem84}]
	Define
	\[S(v,w) = \sum_{\substack{p \mid \frac{vw}{\gcd(v,w)^2}, \\ p \geq s, \,\, p \in \mathcal{R}(G)}} \frac{1}{p} . \]
	Following the steps in~\cite[Proof of Lemma 8.4]{km}, from the assumptions $\mathcal{R}^{\twonotes}(G) = \emptyset$ and $\delta(G) \ge 1/s^{1/4}$ we deduce that
	\[ \sum_{(v,w) \in \mathcal{E}} \mu(v) \mu(w) S(v,w) \leq \sum_{p \geq s}\frac{2 \mu(\mathcal{V})\mu(\mathcal{W})}{p^{3/2}}
	\leq \frac{\mu(\mathcal{E})}{100s^{1/4}} \]
	for large enough $s$. Consider the edge set
	\[\mathcal{E}' := \Bigg\{(v,w) \in \mathcal{E}: S(v,w) \leq \frac{1}{4 s^{1/4}}\Bigg\} . \]
	An application of the Markov inequality gives
	\[ \mu(\mathcal{E}\setminus \mathcal{E'}) \leq 4s^{1/4} \sum_{(v,w) \in \mathcal{E}} \mu(v)\mu(w)S(v,w) \leq \frac{\mu(\mathcal{E})}{25}, \]
	that is, $\mu(\mathcal{E'}) \geq (24/25)\mu(\mathcal{E})$. By the definition of quality, the GCD subgraph $G' := (\mu,\mathcal{V},\mathcal{W},\mathcal{E}',\mathcal{P},f,g)$ thus satisfies
	\[ \frac{q(G')}{q(G)} = \left(\frac{\mu(\mathcal{E}')}{\mu(\mathcal{E})}\right)^{10} \geq \frac{1}{2}. \]
	Further, for any $(v,w) \in \mathcal{E}'$ we have 
	\[\sum_{\substack{p \mid \frac{vw}{\gcd(v,w)^2}, \\ p \geq s, \,\, p \notin \mathcal{R}(G)}}\frac{1}{p} = L_{s}(v,w) - S(v,w) \geq \frac{3}{4s^{1/4}}, \]
	as claimed.
\end{proof}

To prove Lemma~\ref{empty_R}, we will apply the following two propositions in an iterative way.
\begin{prop} \label{prop_iteration1} Let $G=(\mu,\mathcal{V},\mathcal{W},\mathcal{E},\mathcal{P},f,g)$ be a GCD graph with $\delta(G) > 0$. Then there is a GCD subgraph $G'=(\mu, \mathcal{V}', \mathcal{W}', \mathcal{E}', \mathcal{P}', f',g')$ of $G$ such that
\[ \mathcal{P}' \subseteq \mathcal{P} \cup (\mathcal{R}(G) \cap \{ p \le 10^{2000} \}), \quad \mathcal{R}(G') \subseteq \{p > 10^{2000}\}, \quad \frac{q(G')}{q(G)} \geq \frac{1}{10^{10^{3000}}}. \]
\end{prop}

\begin{proof} This is a slight modification of~\cite[Proposition 8.3]{km}, the only difference being that in our formulation the set $\mathcal{P}$ can be non-empty. The proof given in~\cite{km} actually covers the formulation stated above, since it only relies on the iterative application of~\cite[Lemma 13.2]{km}, which holds for GCD graphs with an arbitrary set of primes.
\end{proof}

\begin{prop} \label{prop_iteration2} Let $G=(\mu,\mathcal{V},\mathcal{W},\mathcal{E},\mathcal{P},f,g)$ be a GCD graph with $\delta(G) > 0$ such that $\emptyset \neq \mathcal{R}(G) \subseteq \{p > 10^{2000}\}$. Then there is a GCD subgraph $G'=(\mu, \mathcal{V}', \mathcal{W}', \mathcal{E}', \mathcal{P}', f',g')$ of $G$ such that
\[ \mathcal{P} \subsetneq \mathcal{P}' \subseteq \mathcal{P}\cup \mathcal{R}(G), \quad \mathcal{R}(G') \subsetneq \mathcal{R}(G),\quad q(G') \geq q(G). \]
\end{prop}

\begin{proof} This follows directly from~\cite[Propositions 8.1 and 8.2]{km}.
\end{proof}

\begin{proof}[Proof of Lemma~\ref{empty_R}] First, we apply Proposition~\ref{prop_iteration1} to obtain a GCD subgraph $G^{(1)} \preceq G$ with
\[ \mathcal{R}(G^{(1)}) \subseteq \{p>10^{2000}\} \quad \textrm{and} \quad q(G^{(1)}) \gg q(G). \]
If $\mathcal{R}(G^{(1)}) = \emptyset$, we are done. Otherwise, we apply Proposition~\ref{prop_iteration2} to obtain a GCD subgraph $H_1 \preceq G^{(1)}$ with $\mathcal{R}(H_1) \subsetneq \mathcal{R}(G^{(1)})$ and $q(H_1) \geq q(G^{(1)}).$ By iterating this argument, we obtain a chain of GCD subgraphs $G^{(1)} \succeq H_1 \succeq H_2 \succeq \cdots$ with
\[ \mathcal{R}(G^{(1)}) \supsetneq \mathcal{R}(H_1) \supsetneq \mathcal{R}(H_2) \supsetneq \cdots \quad \textrm{and} \quad q(G^{(1)}) \leq q(H_1) \leq q(H_2) \leq \cdots . \]
Since $\mathcal{R}(G^{(1)})$ is a finite set, we arrive after finitely many steps at a GCD subgraph $G' \preceq G$ with $\mathcal{R}(G') = \emptyset$ and $q(G') \geq q(G^{(1)}) \gg q(G)$. Furthermore, we have $\mathcal{P}' \subseteq \mathcal{P} \cup \mathcal{R}(G)$ since this property is preserved at each step.
\end{proof}

\section{Quality increment vs.\ density loss} \label{sec_9}

The goal of this section is to prove Lemma~\ref{quality_density_lemma}. We start with three preliminary results.
\begin{lemma} \label{lemma121} Let $G=(\mu,\mathcal{V},\mathcal{W},\mathcal{E},\mathcal{P},f,g)$ be a GCD graph with $\delta(G)>0$, let $p \in\mathcal{R}(G)$, and let
\[ \alpha_k= \frac{\mu(\mathcal{V}_{p^k})}{\mu(\mathcal{V})} \qquad \text{and} \qquad \beta_l= \frac{\mu(\mathcal{W}_{p^l})}{\mu(\mathcal{W})}. \]
Then there exists a pair of non-negative integers $(k,l)=(k_p,l_p)$ such that $\alpha_k, \beta_l >0$, and
\[ \frac{\mu(\mathcal{E}_{p^k,p^l})}{\mu(\mathcal{E})} \geq \left\{ \begin{array}{ll} (\alpha_k \beta_k)^{9/10} & \textrm{if } k=l, \\ \frac{\alpha_k (1 - \beta_k) + \beta_k (1-\alpha_k) + \alpha_l (1-\beta_l) + \beta_l (1-\alpha_l)}{40 |k-l|^2} & \textrm{if } k \neq l . \end{array} \right. \]
\end{lemma}

\begin{proof} This follows from a straightforward modification of the proof of~\cite[Lemma 12.1]{km}, replacing the estimate$\frac{1}{1000} \sum_{|j|\geq 1} 2^{-|j|/20} \leq \frac{1}{10}$ by $\sum_{|j| \geq 1} \frac{1}{40 j^2} \leq \frac{1}{10}$ in one of the steps.
\end{proof}

\begin{lemma}\label{optim_2} Let $\alpha_k,\beta_k,\alpha_l,\beta_l \in [0,1]$ with $\alpha_k, \beta_l >0$ be such that $\alpha_k + \alpha_l \leq 1$ and $\beta_k + \beta_l \leq 1$, and let
\[ S = \alpha_k (1 - \beta_k) + \beta_k (1-\alpha_k) + \alpha_l (1-\beta_l) + \beta_l (1-\alpha_l). \]
If $\min\{\alpha_k,\beta_k\} \leq 1-R$ and $\min\{\alpha_{l},\beta_{l}\} \leq 1-R$ with some $R \in \left[0,1/\sqrt{2} \right]$, then $\frac{S^2}{\alpha_k\beta_{l}} \geq \frac{R}{2}$.
\end{lemma}

\begin{proof} Clearly,
\begin{equation} \label{S_larger_ab}
S \geq \alpha_k(1-\beta_k) + \beta_{l}(1-\alpha_{l}) \geq \alpha_k\beta_{l} + \beta_{l}\alpha_k = 2\alpha_k\beta_{l} .
\end{equation}
Since the conditions of the lemma and $S$ are invariant under switching $\alpha_k$ with $\beta_l$ and $\alpha_l$ with $\beta_k$, respectively, we may assume that $\alpha_k \geq \beta_{l}$.

Assume first that $\alpha_k \le 1/2$. Then $\beta_l \le 1/2$ as well, hence
\[ S = \beta_k(1 - 2\alpha_k) + \alpha_k + \alpha_{l}(1 - 2\beta_{l}) + \beta_{l} \geq \alpha_k + \beta_{l} \ge 2 \sqrt{\alpha_k \beta_l} . \]
Therefore $S^2/(\alpha_k \beta_l) \ge 4>R/2$, as claimed.

Assume next that $\alpha_k>1/2$. Formula \eqref{S_larger_ab} then gives
\[ 1 - \beta_k \le 2\alpha_k(1-\beta_k) \leq 2S \leq \frac{S^2}{\alpha_k\beta_{l}} . \]
If $\beta_k \le 1-R$, then $R \le 1-\beta_k \le S^2/(\alpha_k \beta_l)$, as claimed. If $\beta_k>1-R>1/4$, then by the assumption $\min \{ \alpha_k, \beta_k \} \le 1-R$ we have $\alpha_k \le 1-R$, and we similarly deduce
\[ R \le 1-\alpha_k \le 4 \beta_k (1-\alpha_k) \le 4S \le 2 \frac{S^2}{\alpha_k \beta_l}, \]
which finishes the proof of the statement.
\end{proof}

The following lemma is a variant of~\cite[Lemma 12.2]{km}. 
\begin{lemma} \label{lemma122} Consider a GCD graph $G=(\mu,\mathcal{V},\mathcal{W},\mathcal{E},\mathcal{P},f,g)$ with $\delta(G)>0$ and a prime $p \in \mathcal{R}^{\twonotes}(G)$. Let $(k,l)=(k_p,l_p)$ be a pair of non-negative integers which satisfies the conclusion of Lemma~\ref{lemma121}. Then there is a GCD subgraph $G'=(\mu,\mathcal{V}',\mathcal{W}',\mathcal{E}',\mathcal{P}',f',g')$ of $G$ with $\mathcal{P}' = \mathcal{P} \cup \{p\}$ and $\mathcal{R}(G') \subseteq \mathcal{R}(G) \backslash \{p\}$ such that
\[ \frac{\delta(G')}{\delta(G)} \ge \left\{ \begin{array}{ll} 1 & \textrm{if } k=l, \\ \frac{1}{20 |k-l|^2} & \textrm{if } k \neq l, \end{array} \right. \]
and
\[ \frac{q(G')}{q(G)} \ge \left\{ \begin{array}{ll} 1 & \textrm{if } k=l, \\ \frac{p^{|k-l| -1/2}}{10^{15} |k-l|^{20}} & \textrm{if } k \neq l. \end{array} \right. \]
\end{lemma}

\begin{proof} We claim that $G'= G_{p^k, p^l}$ satisfies all required properties. Note that $\mathcal{P}' = \mathcal{P} \cup \{p\}$ and $\mathcal{R}(G') \subseteq \mathcal{R}(G) \backslash \{p\}$ hold by the definition of $G_{p^k, p^l}$. If $k = l$, then by Lemma~\ref{lemma121} and the definition of quality,
\[ \frac{\delta(G')}{\delta(G)} = \frac{\mu(\mathcal{E}_{p^k,p^k})}{\mu(\mathcal{E})} \cdot \frac{1}{\alpha_k\beta_k} \geq 1, \]
and
\[ \frac{q(G')}{q(G)} = \left(\frac{\mu(\mathcal{E}_{p^k,p^k})}{\mu(\mathcal{E})}\right)^{10}(\alpha_k\beta_k)^{-9}\frac{1}{(1 - \mathbbm{1}_{k \geq 1}/p)^2(1 - 1/p^{31/30})^{10}}\geq 1 , \]
as claimed. Let $S$ be as in Lemma~\ref{optim_2}. If $k \neq l$, then by Lemma~\ref{lemma121} together with \eqref{S_larger_ab},
\[ \frac{\delta(G')}{\delta(G)} = \frac{\mu(\mathcal{E}_{p^k,p^l})}{\mu(\mathcal{E})} \cdot \frac{1}{\alpha_k \beta_l} \geq \frac{S}{40 |k-l|^2 \alpha_k \beta_l} \geq  \frac{1}{20 |k-l|^2}.\]
Furthermore,
\[ \begin{split} \frac{q(G')}{q(G)} = \left(\frac{\mu(\mathcal{E}_{p^k,p^l})}{\mu(\mathcal{E})}\right)^{10}(\alpha_k\beta_l)^{-9} \frac{p^{|k-l|}}{(1 - 1/p^{31/30})^{10}} &\geq \frac{S^{10}}{(40|k-l|^2)^{10}} \cdot \frac{1}{(\alpha_k \beta_l)^9} p^{|k-l|} \\ &\geq \frac{2^8 p^{|k-l|}}{40^{10} |k-l|^{20}} \cdot \frac{S^2}{\alpha_k \beta_l} . \end{split} \]
The assumption $p \in \mathcal{R}^{\twonotes}(G)$ ensures that $\min \{ \alpha_k, \beta_k \} \leq 1 -1/\sqrt{p}$ and $\min \{ \alpha_l, \beta_l \} \leq 1 - 1/\sqrt{p}$. Hence we can apply Lemma~\ref{optim_2} with $R = 1/\sqrt{p}$, which shows that
\[ \frac{q(G')}{q(G)} \geq \frac{2^7 p^{|k-l|-1/2}}{40^{10}|k-l|^{20}} > \frac{p^{|k-l|-1/2}}{10^{15}|k-l|^{20}}, \]
as claimed.
\end{proof}

\begin{proof}[Proof of Lemma~\ref{quality_density_lemma}] We apply Lemma~\ref{lemma122} iteratively to $G$ until we obtain a GCD subgraph $G'=(\mu, \mathcal{V}', \mathcal{W}', \mathcal{E}', \mathcal{P}', f', g')$ of $G$ such that $\mathcal{R}^{\twonotes}(G')=\emptyset$. Note that each prime $p$ is used at most once, and $\mathcal{P}'$ is precisely the set of primes to which Lemma~\ref{lemma122} was applied. For each $p \in \mathcal{P}'$, let $(k_p,l_p)$ be the pair of non-negative integers with which Lemma~\ref{lemma122} is applied.\footnote{We might use primes $p \not\in \mathcal{R}^{\twonotes}(G)$ of the original GCD graph $G$, since $\mathcal{R}^{\twonotes}$ does not necessarily decrease at each step. However, $\mathcal{R}$ decreases by at least one element at each step, hence the algorithm terminates.} Since the original graph $G$ had an empty set of primes, we have $\mathcal{P}_{\textrm{diff}}(G') = \{ p \in \mathcal{P}' \, : \, k_p \neq l_p \}$. By Lemma~\ref{lemma122}, the resulting graph $G'$ satisfies
\[ \frac{\delta(G')}{\delta(G)} \ge \prod_{p \in \mathcal{P}_{\textrm{diff}}(G')} \frac{1}{20 |k_p-l_p|^2} \quad \textrm{and} \quad \frac{q(G')}{q(G)} \ge \prod_{p \in \mathcal{P}_{\textrm{diff}}(G')} \frac{p^{|k_p-l_p|-1/2}}{10^{15} |k_p-l_p|^{20}} . \]
In particular,
\begin{equation}\label{qG'/qGbound}
\frac{q(G')}{q(G)} \gg \prod_{p \in \mathcal{P}_{\textrm{diff}}(G')} p^{|k_p-l_p|/4} \gg 1.
\end{equation}

Fix $C \ge 1$, and let $t \ge 1$ be large enough in terms of $C$. Let $N=|\mathcal{P}_{\textrm{diff}}(G')|$, and for the sake of readability, in the sequel let $\log_i$ denote the $i$-fold iterated logarithm. It will be enough to show that if $q(G')<t^3 q(G)$ (i.e.\ property a) does not hold), then $\delta(G')/\delta(G) \ge 1/F(t)^{1/4}$, and $N \le \log t$ (i.e.\ property b) holds). The latter follows easily from \eqref{qG'/qGbound} and $q(G')<t^3 q(G)$:
\[ (N!)^{1/4} \le \prod_{p \in \mathcal{P}_{\textrm{diff}}(G')} p^{|k_p-l_p|/4} \ll \frac{q(G')}{q(G)} < t^3. \]
Hence $N \ll (\log t)/\log_2 t$, and in particular, $N \le \log t$ for large enough $t$, as claimed. It remains to show that $q(G')<t^3 q(G)$ implies $\delta(G')/\delta(G) \ge 1/F(t)^{1/4}$.

Let $Y=\{ p \in \mathcal{P}_{\mathrm{diff}}(G') \, : \, |k_p-l_p| \ge \log_3 t \}$. Bounding the sum term by term gives
\begin{equation}\label{pnotinY}
\sum_{p \not\in Y} \log (20 |k_p-l_p|^2) \ll N \log_4 t \ll \frac{\log t \log_4 t}{\log_2t } .
\end{equation}
On the other hand, \eqref{qG'/qGbound} and the assumption $q(G')<t^3 q(G)$ lead to
\[ \log t \gg \sum_{p \in \mathcal{P}_{\mathrm{diff}}(G')} |k_p-l_p| \log p \ge \log_3 t \sum_{p \in Y} \log p \gg (\log_3 t) |Y| \log |Y|, \]
hence $|Y| \ll (\log t)/(\log_2 t \log_3 t)$. The previous formula also shows that $\sum_{p \in Y} |k_p-l_p| \ll \log t$. An application of the inequality of arithmetic and geometric means thus yields
\[ \begin{split} \sum_{p \in Y} \log (20 |k_p-l_p|^2) \le 2 \sum_{p \in Y} \log (20 |k_p-l_p|) &\le 2|Y| \log \frac{\sum_{p \in Y} 20|k_p-l_p|}{|Y|} \\ &\ll |Y| \log \left(\frac{\log t}{|Y|}\right) \\ &\ll \frac{\log t}{\log_2 t}. \end{split} \]
The previous formula and \eqref{pnotinY} thus give
\[ -\log \frac{\delta(G')}{\delta(G)} \le \sum_{p \in \mathcal{P}_{\mathrm{diff}}(G')} \log (20|k_p-l_p|^2) \ll \frac{\log t}{\log_2 t} . \]
Hence $-\log (\delta(G')/\delta(G)) \le (1/4) \log F(t)$ for large enough $t$, that is, $\delta(G')/\delta(G) \ge 1/F(t)^{1/4}$, and we obtain the desired result.
\end{proof}

\section*{Acknowledgments}

CA is supported by the Austrian Science Fund (FWF), projects F-5512, I-3466, I-4945, I-5554, P-34763, P-35322 and Y-901. BB is supported by the Austrian Science Fund (FWF), project F-5510.


\end{document}